\documentclass[12pt]{article}
\usepackage[margin=1in, a4paper]{geometry}

\usepackage{amsmath,enumerate,amsfonts,amssymb,color,graphicx,amsthm}

\usepackage{hyperref}

\usepackage{mathrsfs}

\def\RR{{\mathbb R}}

\def\SSphere{{\mathbb S}}

\newcounter{marnote}

\newtheorem{thm}{Theorem}[section]

\newtheorem{lem}[thm]{Lemma}

\newtheorem{prop}[thm]{Proposition}

\numberwithin{equation}{section}

\begin{document}

\title{Two-bump axisymmetric solutions\\ of the Nirenberg problem}

\author{YanYan Li \thanks{Department of Mathematics, Rutgers University, Hill Center, Busch Campus, 110 Frelinghuysen Road, Piscataway, NJ 08854, USA. Email: yyli@math.rutgers.edu.}~\thanks{Partially supported by NSF Grant DMS-2247410.}~ and Luc Nguyen \thanks{Mathematical Institute and St Edmund Hall, University of Oxford, Andrew Wiles Building, Radcliffe Observatory Quarter, Woodstock Road, Oxford OX2 6GG, UK. Email: luc.nguyen@maths.ox.ac.uk.}~ and Bo Wang \thanks{School of Mathematics and Statistics, Beijing Institute of Technology, No 5 Zhongguancun South Street, Haidian District, Beijing 100081, China. Corresponding author. Email: wangbo89630@bit.edu.cn.}~\thanks{Partially supported by NNSF 12271028.}}

\date{\small Dedicated to Jean-Michel Coron on the occasion of his 70th birthday
}

\maketitle

\begin{abstract}
In this paper, we revisit the axisymmetric Nirenberg problem for certain prescribed scalar curvature functions whose critical points include the north and south poles and whose order of flatness at both points is $n-2$. It has been known for some time that the solution set is compact when the parameter
points $(K_1,K_2,a_1,a_2)\in\mathbb R^4$, determined by the first two leading expansions of the prescribed scalar curvature functions  at the poles, stay away from a critical three-dimensional hypersurface $\Sigma^0$, and that loss of compactness occurs near $\Sigma^0$. 
We construct, in dimensions $n \geq 4$, two-bump axisymmetric solutions for parameter points sufficiently close to $\Sigma^0$ from one side, where the relevant side is determined by the next-order flatness terms. This gives a one-sided refinement of the previously known blow-up phenomenon near the critical hypersurface $\Sigma^0$: two-bump blow-up occurs from this side, whereas the solution set remains compact on the opposite side.
\end{abstract}

\section{Introduction}

Let $n\geq2$ and $(\SSphere^{n},g_{0})$ be the standard $n$-sphere. In the years 1969--1970, L. Nirenberg raised the following question: Which function $K$ on $\SSphere^{2}$ can arise as the Gauss curvature of a metric $g$ on $\SSphere^{2}$ conformally equivalent to $g_{0}$? Writing $g=e^{2v}g_{0}$, the problem is equivalent to finding a function $v$ on $\SSphere^{2}$ satisfying
\begin{equation}\label{2d}
-\Delta_{g_{0}}v+1=Ke^{2v}~~\mbox{on }\SSphere^{2},
\end{equation}
where $\Delta_{g_{0}}$ denotes the Laplace--Beltrami operator associated with the metric $g_{0}$.

It is natural to
 extend Nirenberg's question to higher-dimensional spheres $\SSphere^{n}$, $n\geq3$. Writing $g=u^{4/(n-2)}g_{0}$, the  problem is equivalent to finding a positive function $u$ on $\SSphere^{n}$
 satisfying
\begin{equation}\label{main equ}
-L_{g_{0}}u:=-\Delta_{g_{0}}u+c(n)R_{0}u=c(n)Ku^{\frac{n+2}{n-2}}~~\mbox{on }\SSphere^{n},
\end{equation}
where $c(n)=\frac{n-2}{4(n-1)}$ and $R_{0}=n(n-1)$ is the scalar curvature of $g_{0}$.

The first work on the Nirenberg problem was due to
D. Koutroufiotis \cite{K},
who established the solvability of (\ref{2d}) when $K$
is an antipodally symmetric function sufficiently close to
 $1$. J. Moser \cite{M} subsequently proved solvability for all antipodally symmetric positive functions
  $K$. Without any symmetry assumption on $K$, the groundbreaking works 
of A. Chang and P. Yang \cite{CY1,CY2} and of
A. Bahri and J. M. Coron \cite{BC} provided  sufficient conditions for solvability 
 in dimensions $n=2$ and  $n=3$,  respectively. 
 Compactness of the set of all solutions in dimensions $n=2$, $3$ can be found in  \cite{CGY},  \cite{H}, and  \cite{SZ}. In these dimensions, a sequence of solutions cannot blow up at more than one point.

Compactness and existence of solutions in dimensions $n\geq 4$ were established
by the first named author in \cite{L95,L96}, in the spirit of the work of
Schoen and Zhang \cite{SZ} in dimension $n=3$. A new phenomenon arises in
dimensions
$n\ge 4$: unlike the cases $n=2,3$, a sequence of solutions may
blow up at more than one point. Such multiple-point blow-up was first
established in \cite{L96} for suitable $K$. Later, for suitable $K$, unbounded-energy blow-up in dimensions
$n\ge 7$ was proved 
 in \cite{CL2}.

In this paper, we work with axisymmetric functions $K$ having the south and north poles, $P_1$ and $P_2$, as critical points of
flatness order $n-2$. 
We 
construct two-bump axisymmetric solutions $u$ of (\ref{main equ}), with their mass concentrated near $P_1$ and $P_2$.

We view
$(\SSphere^{n}, g_{0})=\{(x^{1})^{2}+\cdots+(x^{n+1})^{2}=1\}$ as the unit sphere embedded in $\RR^{n+1}$.  Then    both  $K$ and $u$ depend only on $\theta :=\arccos x^{n+1}\in[0,\pi]$. For any integer $m\geq0$, let $C^{m}_{r}(\SSphere^{n})$ denote the space of $C^{m}$ axisymmetric functions on $\SSphere^{n}$. Given a positive axisymmetric function $K$ on $\SSphere^{n}$, define 
\begin{equation*}
\Lambda(K):=\sup\{\|u\|_{L^{\infty}(\SSphere^{n})}:\mbox{$u$ is an axisymmetric $C^{2}$ positive solution of (\ref{main equ})}\}.
\end{equation*}
If equation (\ref{main equ}) admits no
axisymmetric   $C^{2}$ positive solution 
for this $K$,
we set $\Lambda(K)=0$. 

In \cite{L96}, it was shown that the solution set is compact when the parameter points determined by the first two leading expansions of $K$ at the poles stay away from a certain codimension-one hypersurface $\Sigma^0$, while loss of compactness can occur as a sequence of  parameter points approaches $\Sigma^0$. We study the behavior near $\Sigma^0$ more precisely and show that the next-order terms in the expansion of $K$ at the poles determine one side of $\Sigma^0$ on which two-bump blow-up occurs, while the solution set remains compact on the opposite side.

Before stating our main result, let us recall 
the two-point blow-up phenomenon in dimensions
 $n\geq 4$, as established  in \cite{L96}. We write
\begin{equation*}
\nu=(K_{1},K_{2},-a_{1},-a_{2})\in\RR^{4}_{+}.
\end{equation*}
Define
\begin{align*}
\Sigma^{0}&:=\{\nu\in\RR^{4}_{+}:(K_{1}K_{2})/(a_{1}a_{2})=\bar{c}(n)\},\\
\Sigma^{+}&:=\{\nu\in\RR^{4}_{+}:(K_{1}K_{2})/(a_{1}a_{2})>\bar{c}(n)\},\\
\Sigma^{-}&:=\{\nu\in\RR^{4}_{+}:(K_{1}K_{2})/(a_{1}a_{2})<\bar{c}(n)\},
\end{align*}
where $\bar{c}(n)= \frac{2^{2(n-3)}(n-2)^2}{(n-1)^2}$.

Fix a non-negative non-decreasing continuous function $\omega: [0,\pi] \rightarrow [0,\infty)$ with $\omega(0) = 0$. For $\nu=(K_{1},K_{2},-a_{1},-a_{2})\in \RR^4_+$, let $\hat{\Sigma}_\nu$ denote the set of positive functions $K\in C_r^1(\SSphere^n)$ satisfying 
\begin{align*}
K(\theta)&=K_{1}+a_{1}
(\pi-\theta)^{n-2}+R^{(1)}(\theta)\\
&=K_{2}+a_{2}\theta^{n-2}+R^{(2)}(\theta),
\end{align*}
where 
\begin{equation}
    \label{higher0}
\frac{|R^{(1)}(\theta)|+|\pi-\theta||R^{(1)}(\theta)'|}{|\pi-\theta|^{n-2}} \leq \omega(\pi - \theta) \quad
 \text{ and } \quad \frac{|R^{(2)}(\theta)|+|\theta||R^{(2)}(\theta)'|}{|\theta|^{n-2}} \leq \omega(\theta).
\end{equation}
Define
$$
\hat \Sigma^0:= \bigcup\limits_{ \nu\in \Sigma^0} \hat \Sigma_\nu, \qquad
\hat \Sigma^+:= \bigcup\limits_{ \nu\in \Sigma^+} \hat \Sigma_\nu, \qquad
\hat \Sigma^-:= \bigcup\limits_{ \nu\in \Sigma^-} \hat \Sigma_\nu.
$$

It was proved in \cite[Theorem 0.20]{L96} that, for each  $K\in\hat\Sigma^{+}$, there exists a positive constant $C_{1}$ such that  any 
positive solution  $u\in C_{r}^{2}(\SSphere^n)$
of (\ref{main equ}) satisfies
\begin{equation*}
1/C_{1}\leq u\leq C_{1},~~\|u\|_{C^{2,\alpha}(\SSphere^n)}\leq C_{1}.
\end{equation*}
Moreover, for any $C>C_{1}$, the degree 
\begin{align*}
\mbox{Index}(K)&:=\mbox{deg}\Big(u+L_{g_{0}}^{-1}(c(n)Ku^{(n+2)/(n-2)}),\\
&\quad\quad\quad\quad\big\{u\in C_{r}^{2}(\SSphere^n),1/C<u<C,~~\|u\|_{C^{2,\alpha}(\SSphere^n)}<C\big\},0\Big)   
\end{align*}
is well defined and is equal to $-1$. Similarly, if $K\in\hat \Sigma^{-}$, $\mbox{Index}(K)$ is well defined and is equal to $0$. (This result was generalized by the authors in \cite{LNW} to the
$\sigma_{k}$-Nirenberg problem for $2\leq k<n/2$.)

Let
\[
\gamma:\tau\in[-1,1]\longmapsto
\bigl(K_{1}(\tau),K_{2}(\tau),-a_{1}(\tau),-a_{2}(\tau)\bigr)\in\RR^{4}_{+}
\]
be a continuous map such that
\[
\gamma(0)\in\Sigma^{0},\qquad
\gamma(\tau)\in\Sigma^{+}\ \text{for }0<\tau\leq1,\qquad
\gamma(\tau)\in\Sigma^{-}\ \text{for }-1\leq \tau<0.
\]
Let $\{K_\tau\}_{\tau\in[-1,1]}$ be a family of positive functions with $K_\tau \in \hat\Sigma_{\gamma(\tau)}$. Then
\begin{equation*}
    \mbox{Index}(K_\tau)=
    \begin{cases}
        -1& \text{ if } 0<\tau<1,\\
        0& \text{ if }-1<\tau<0.
    \end{cases}
\end{equation*}
Consequently, by the homotopy invariance of the Leray-Schauder degree, for every $0<\varepsilon<1$ one must have 
\begin{equation}\label{0000}
\sup\limits_{0<|\tau|<\varepsilon}\Lambda(K_{\tau})=\infty.
\end{equation}
See \cite[Corollary 0.24]{L96}.

In the present paper, we strengthen and generalize \eqref{0000}. In particular, our result gives a sharper description of the blow-up behavior of $\Lambda(K_\tau)$ near $\tau=0$. While \eqref{0000} only yields a sequence $\{\tau_i\}$, with $\tau_i\to0$, such that $\Lambda(K_{\tau_i})\to\infty$, our result determines the one-sided behavior such as 
\[
\lim_{\tau\to0^{+}}\Lambda(K_\tau)=\infty \quad \text{ and } \quad \limsup_{\tau\to0^{-}}\Lambda(K_\tau)<\infty, 
\]
or
\[
\limsup_{\tau\to0^{+}}\Lambda(K_\tau) < \infty \quad \text{ and } \quad \lim_{\tau\to0^{-}}\Lambda(K_\tau) = \infty.
\]
Thus the blow-up is localized to one side of $\Sigma^{0}$, rather than merely detected along an unspecified sequence.

To determine the side of $\Sigma^0$ on which blow-up actually occurs, one makes use of the next non-vanishing terms in the expansions of $K$ at the two poles. More precisely, we assume $K$ has the form
\begin{align*}
K(\theta)
    &=
K_{1}+a_{1}(\pi-\theta)^{n-2}+b_{1}(\pi-\theta)^{\mu_{1}}+ \ldots\nonumber\\
&=K_{2}+a_{2}\theta^{n-2}+b_{2}\theta^{\mu_{2}}+ \ldots,
\end{align*}
where \begin{equation*}
\mu=(\mu_{1},\mu_{2})\in(n-2,n)^    {2},~~b=(b_{1},b_{2})\in\RR^{2},~~b_{1}b_{2}\neq0,
\end{equation*}
and the $\ldots$ denotes error terms. We will see that, as $\nu$ approaches a fixed $\bar \nu \in \Sigma^0$, then
\begin{itemize}
    \item when $\mu_1 \neq \mu_2$, the smaller exponent dominates compactness properties;
    \item when $\mu_1 = \mu_2$, the compactness of the solution set depends on the sign of the quantity
\begin{equation*}
    F_{\mu,b}(\nu):=b_{1}(-a_{2})^{\frac{\mu_{1}}{n-2}}K_{2}^{-\frac{n(\mu_{1}-(n-2))}{2(n-2)}}+b_{2}(-a_{1})^{\frac{\mu_{1}}{n-2}}K_{1}^{-\frac{n(\mu_{1}-(n-2))}{2(n-2)}}.
\end{equation*}
\end{itemize}
In particular, the first non-vanishing next-order terms of $K$ at the poles determine the side of $\Sigma^0$ on which blow-up occurs. 

In view of the above and analogous to $\Sigma^0$, $\Sigma^\pm$, we define in case $\mu_1 = \mu_2$ the sets
\begin{align*}
S^{0}_{\mu, b}
    &:=\{\nu\in\RR^{4}_{+}:
F_{\mu, b}(\nu)=0\},\\
S^{+}_{\mu, b} &:=\{\nu\in\RR^{4}_{+}: F_{\mu, b}(\nu)>0\},\\
S^{-}_{\mu, b}
    &:=\{\nu\in\RR^{4}_{+}:F_{\mu,b}(\nu)<0\}.
\end{align*}
Since $\nabla F_{\mu,b}(\nu)\neq0$ whenever $\nu\in\RR^{4}_{+}$ and $F_{\mu,b}(\nu)=0$, the implicit function theorem implies that $S^{0}_{\mu, b}$ is a smooth hypersurface in $\RR^{4}_{+}$. 

We also introduce, for $\mu,b,\nu$ as above, the set $\hat\Sigma_{\mu,b,\nu}$ of positive functions $K \in C^1_r(\mathbb{S}^n)$ such that
\begin{align*}
K(\theta)
    &=
K_{1}+a_{1}(\pi-\theta)^{n-2}+b_{1}(\pi-\theta)^{\mu_{1}}+R^{(1)}(\theta)\nonumber\\
&=K_{2}+a_{2}\theta^{n-2}+b_{2}\theta^{\mu_{2}}+R^{(2)}(\theta)
\end{align*}
where the error terms $R^{(i)}$ are assumed to satisfy
\begin{equation}\label{higher0X}
\frac{|R^{(1)}(\theta)|+|\pi-\theta||R^{(1)}(\theta)'|}{|\pi-\theta|^{\mu_{1}}} \leq \omega(\pi - \theta)
\quad\text{ and }\quad \frac{|R^{(2)}(\theta)|+|\theta||R^{(2)}(\theta)'|}{|\theta|^{\mu_{2}}} \leq \omega(\theta).
\end{equation}

The following table gives a summary of the notations introduced above.

\medskip
\begin{center}
\begin{tabular}{l|l}
    Object & Meaning\\
    \hline
    $\nu$ & The first two leading coefficients of $K$ at the poles\\
    $\Sigma^0$ & Critical hypersurface in parameter space\\
    $\Sigma^\pm$ & Two sides of $\Sigma^0$\\
    \hline
    $\mu$ & The next-order flatness exponent of $K$ at the poles\\
    $b$ & The next-order coefficients of $K$ at the poles\\
    $\hat \Sigma_{\mu,b,\nu}$ & Set of $K$'s whose three leading terms at the poles are given by $\nu$, $b$ and $\mu$\\
    $F_{\mu,b}$ & A function whose sign determines compactness when $\mu_1 = \mu_2$\\
    $S_{\mu,b}^0$ & Zero set of $F_{\mu,b}$ when $\mu_1 = \mu_2$\\
    $S_{\mu,b}^\pm$ & Two sides of $S_{\mu,b}^0$ when $\mu_1 = \mu_2$\\
\end{tabular}
\label{table1}
\end{center}

We are now ready to state our main result.

\begin{thm}\label{main thm}

Let $n\geq4$ and $\bar{\nu}\in\Sigma^{0}$. 
\begin{enumerate}[(a)]
\item In case
\begin{align*}
&(i)~\mu_{1}<\mu_{2},~~b_{1}>0, \text{ or }\\
&(ii)~\mu_{1}>\mu_{2},~~b_{2}>0, \text{ or }\\
&(iii)~\mu_{1}=\mu_{2},~~\bar{\nu}\in S^{+}_{\mu, b},
\end{align*}
we have
\begin{align*}
\lim\limits_{\varepsilon\rightarrow0^{+}}\inf\limits_{\nu\in B_{\varepsilon}(\bar{\nu})\setminus\{\bar{\nu}\}\cap\Sigma^{-}, K \in \hat \Sigma_{\mu,b,\nu}}\Lambda(K)&=\infty,\\
\limsup\limits_{\varepsilon\rightarrow0^{+}}\sup\limits_{\nu\in B_{\varepsilon}(\bar{\nu})\setminus\{\bar{\nu}\}\cap \Sigma^{+}, K \in \hat \Sigma_{\mu,b,\nu}}\Lambda(K)&<\infty.
\end{align*}

\item In the opposite case that
\begin{align*}
&(i)~\mu_{1}<\mu_{2},~~b_{1}<0, \text{ or }\\
&(ii)~\mu_{1}>\mu_{2},~~b_{2}<0,\text{ or }\\
&(iii)~\mu_{1}=\mu_{2},~~\bar{\nu}\in S^{-}_{\mu, b},
\end{align*}
we have
\begin{align*}
\lim\limits_{\varepsilon\rightarrow0^{+}}\inf\limits_{\nu\in B_{\varepsilon}(\bar{\nu})\setminus\{\bar{\nu}\}\cap\Sigma^{+}, K \in \hat \Sigma_{\mu,b,\nu}}\Lambda(K)&=\infty,\\
\limsup\limits_{\varepsilon\rightarrow0^{+}}\sup\limits_{\nu\in B_{\varepsilon}(\bar{\nu})\setminus\{\bar{\nu}\}\cap \Sigma^{-}, K \in \hat \Sigma_{\mu,b,\nu}}\Lambda(K)&<\infty.
\end{align*}
\end{enumerate}
\end{thm}

We briefly comment on our strategy of proof. When $K\equiv R_{0}$, all the solutions of (\ref{main equ}) have been classified (see \cite{O}, \cite{GNN} and \cite{CGS}) and are given by:
\begin{equation}\label{deltaPtexpress}
\delta_{P,t}(x)=\left(\frac{t}{1+\frac{t^{2}-1}{2}(1-P\cdot x)}\right)^{\frac{n-2}{2}},~~~~\forall~x\in\SSphere^{n},
\end{equation}
where $t>0$ and $P\in\SSphere^{n}$. We look for a solution of \eqref{main equ} of the form
\[
u = \sum_{i= 1}^2 \alpha_i \delta_{P_i,t_i} + v
\]
where $v$ is a correction term satisfying the usual orthogonality properties. We first perform a Lyapunov-Schmidt reduction to obtain for fixed $t = (t_1,t_2)$ the pair $(\alpha(t),v(t))$ such that the projection of \eqref{main equ} along the directions of $\alpha$ and $v$ are satisfied -- See Proposition \ref{infinite}. Equation \eqref{main equ} then reduces to a system $H(t) = 0$ of two equations for the two unknowns $t = (t_1,t_2)$. Near $\Sigma^0$, the leading-order part of this system becomes singular, and the next-order data $\mu$ and $b$ determine the solvability of the system. More concretely, but still schematically, this system can be put in the form
\[
A_\nu \begin{pmatrix}t_1^{\frac{n-2}{2}}\\t_2^{\frac{n-2}{2}}\end{pmatrix}
    = \begin{pmatrix}o(1)t_1^{\frac{n-2}{2}}\\o(1)t_2^{\frac{n-2}{2}}\end{pmatrix}
\]
where $A_\nu$ is some constant matrix and the little o notation is meant for large $t_1, t_2$. The matrix $A_\nu$ has the property that it is non-singular if and only if $\nu \notin \Sigma^0$. Thus, as $\nu \rightarrow \bar \nu$, the invertibility of $A_\nu$ degenerates. We proceed by introducing rescaled variables, $\beta$ to capture the common scale of the bubbles and $s$ to capture the relative strength of one bubble with respect to the other, and analyze the new system using the structure of the right hand side given by the next-order term in the expansion of $K$ near the poles -- See Subsection \ref{SSec:Sol}.

\section{Proof}

Fix $\bar \nu \in \Sigma^0$ and consider a function $K \in \hat\Sigma_{\mu,b,\nu}$ with $\nu$ close to $\bar \nu$. Recall that this means that $K$ is positive, belongs to $C^1_r(\mathbb{S}^n)$ and
\begin{align*}
K(\theta)&=K_{1}+a_{1}(\pi-\theta)^{n-2}+b_{1}(\pi-\theta)^{\mu_{1}}+R^{(1)}(\theta)\nonumber\\
&=K_{2}+a_{2}\theta^{n-2}+b_{2}\theta^{\mu_{2}}+R^{(2)}(\theta),  
\end{align*}
where $K_{1}$, $K_{2}>0$, $-a_{1}$, $-a_{2}>0$, $\mu_{1}$, $\mu_{2}\in(n-2,n)$, $b_{1}$, $b_{2}\in\RR$, $b_{1}b_{2}\neq0$, and \eqref{higher0X} holds.

\subsection{A finite-dimensional reduction}

Let $H^{1}_{r}(\SSphere^{n})$ denote the space of axisymmetric functions on $\SSphere^{n}$ with the $H^{1}$ inner product and norm by
\begin{equation*}
\langle u,v\rangle=-\int_{\SSphere^{n}}(L_{g_{0}}u)v,~~~~\|u\|=\sqrt{\langle u,u\rangle}.
\end{equation*}

For $i=1$, $2$, let $\bar{\alpha}_{i}=(R_{0}/K_i)^{(n-2)/4}$. For any $\varepsilon>0$, define
\begin{equation*}
\Omega_{\varepsilon}:=\Omega_{\varepsilon, K_1, K_2}:=\{(t,\alpha)\in\RR^{2}_{+}\times\RR^{2}_{+}:t_{i}>1/\varepsilon,~~|\alpha_{i}-\bar{\alpha}_{i}|<\varepsilon,~~i=1,2\},
\end{equation*}
and
\begin{equation*}
{\cal N}^{2,\varepsilon}_{r}
:={\cal N}^{2,\varepsilon, K_1, K_2}_{r}:=\Big\{u\in H^{1}_{r}(\SSphere^{n}):\exists~ (t,\alpha)\in\Omega_{\varepsilon} \mbox{ such that }\Big\|u-\sum\limits_{i=1}^{2}\alpha_{i}\delta_{P_{i},t_{i}}\Big\|<\varepsilon\Big\}.
\end{equation*}

For any $\varepsilon>0$ and any $u\in {\cal N}^{2,\varepsilon}_{r}$, consider the minimization problem
\begin{equation}
\min\limits_{(t,\alpha)\in\Omega_{\varepsilon}}\Big\|u-\sum\limits_{i=1}^{2}\alpha_{i}\delta_{P_{i},t_{i}}\Big\|.\label{MP}
\end{equation}
Then we have the following lemma, see also \cite[Proposition 2]{BC} and \cite[Proposition 4.1]{L93}.
\begin{lem}\label{xiayongquan}
For any $0<\varepsilon<1$, there exists $\hat{\varepsilon}>0$ depending on $\varepsilon$, the positive lower and upper bound of $K_{1}$ and $K_{2}$ such that, if $\|u-\sum\limits_{i=1}^{2}\alpha_{i}\delta_{P_{i},t_{i}}\|<\hat{\varepsilon}$ for some $(\alpha,t)\in\Omega_{\frac{\varepsilon}{4}}$, the minimization problem (\ref{MP}) has a unique minimum, and it is achieved in $\Omega_{\frac{\varepsilon}{2}}$. Moreover,
\begin{equation}\label{rep}
v:=u-\sum\limits_{i=1}^{2}\alpha_{i}\delta_{P_{i},t_{i}}\in E_{t},
\end{equation}
where
\begin{equation*}
E_{t}:=\Big\{v\in H_{r}^{1}(\SSphere^{n}):\langle v,\delta_{P_{i},t_{i}}\rangle=0,~~\Big\langle v,\frac{\partial \delta_{P_{i},t_{i}}}{\partial t_{i}}\Big\rangle=0,~~i=1,2\Big\}.
\end{equation*}
\end{lem}

Consider the functional
\begin{equation*}
I[u]:=\frac{1}{2}\langle u,u\rangle-c(n)\frac{n-2}{2n}\int_{\SSphere^{n}}K|u|^{\frac{2n}{n-2}},~~\forall~u\in H^{1}_{r}(\SSphere^{n}).
\end{equation*}
It is easy to see that $I$ is $C^{2}$ and the critical points of $I$ in $H^{1}_{r}(\SSphere^{n})$ are solutions of
\begin{equation}\label{liuxiyij}
I'[u]=-L_{g_{0}}u-c(n)K|u|^{\frac{4}{n-2}}u=0.
\end{equation}
It follows from Lemma \ref{xiayongquan} that the triplet $(t,\alpha,v)$ is a good parameterization of ${\cal N}^{2,\varepsilon}_r$ in a neighborhood of $\Big\{\sum\limits_{i=1}^{2}\alpha_{i}\delta_{P_{i},t_{i}}:(t,\alpha,v)\in\Omega_{\varepsilon/2}\Big\}$. Moreover,
\begin{equation*}
\mbox{$I'[u]=0$~~~~if and only if~~~~$\nabla J(t,\alpha,v)=0$},
\end{equation*}
where $J(t,\alpha,v):=I[u]$.

For $t=(t_{1},t_{2})\in\RR_{+}^{2}$ fixed, we first solve the preliminary equations for $\alpha$ and $v$ (depending on $t$):
\begin{align}
F_{1}(\alpha,v)&:=\frac{\partial J}{\partial \alpha_{1}}(t,\alpha,v)=0,\label{F1def}\\
F_{2}(\alpha,v)&:=\frac{\partial J}{\partial \alpha_{2}}(t,\alpha,v)=0,\label{F2def}\\
F_{3}(\alpha,v)&:=\frac{\partial J}{\partial v}(t,\alpha,v)=0.\label{F3def}
\end{align}
In view of $v\in E_{t}$, we have that
 \begin{equation*}
J(t,\alpha,v)=\frac{1}{2}c_{0}|\alpha|^{2}+\alpha_{1}\alpha_{2}\langle\delta_{P_{1},t_{1}},\delta_{P_{2},t_{2}}\rangle+\frac{1}{2}\|v\|^{2}-c(n)\frac{n-2}{2n}\int_{\SSphere^{n}}K|u|^{\frac{2n}{n-2}},
\end{equation*}
where $c_{0}=\langle\delta_{P_{1},t_{1}},\delta_{P_{1},t_{1}}\rangle=\langle\delta_{P_{2},t_{2}},\delta_{P_{2},t_{2}}\rangle=2^{n}c(n)R_{0}|\SSphere^{n-1}|\int_{0}^{\infty}\frac{r^{n-1}dr}{(1+r^{2})^{n}}$ (see (\ref{AB1-}) in Lemma \ref{A1}).

Then the expressions of $F_{1}$, $F_{2}$ and $F_{3}$ are as follows.
\begin{align*}
F_{1}(\alpha,v)&=c_{0}\alpha_{1}+\alpha_{2}\langle\delta_{P_{1},t_{1}},\delta_{P_{2},t_{2}}\rangle-c(n)\int_{\SSphere^{n}}K|u|^{\frac{4}{n-2}}u\delta_{P_{1},t_{1}},\\
F_{2}(\alpha,v)&=c_{0}\alpha_{2}+\alpha_{1}\langle\delta_{P_{1},t_{1}},\delta_{P_{2},t_{2}}\rangle-c(n)\int_{\SSphere^{n}}K|u|^{\frac{4}{n-2}}u\delta_{P_{2},t_{2}},
\end{align*}
while for $\varphi\in E_{t}$,
\begin{equation*}
F_{3}(\alpha,v)[\varphi]=\frac{d}{ds}|_{s=0}J[t,\alpha,v+s\varphi]=\langle v,\varphi\rangle-c(n)\int_{\SSphere^{n}}K|u|^{\frac{4}{n-2}}u\varphi.
\end{equation*}
It follows that in the space $E_{t}$,
\begin{equation*}
F_{3}(\alpha,v)=v-\Pi_{t}(-L_{g_{0}})^{-1}(c(n)K|u|^{\frac{4}{n-2}}u),
\end{equation*}
where $\Pi_{t}$ denotes the orthogonal projection map of $H^{1}_{r}(\SSphere^{n})$ onto $E_{t}$.

Setting
\begin{equation}\label{Fdef}
F(\alpha,v):=(F_{1}(\alpha,v),F_{2}(\alpha,v),F_{3}(\alpha,v))
\end{equation}
for fixed $t$, with the aid of the implicit function theorem, we solve
\begin{equation*}
F(\alpha,v)=0.
\end{equation*}
We shall, as in \cite{L97,LN}, make use of the following easily verified form of the implicit function theorem:
\begin{lem}\label{ift}
Let $X$, $Y$ be Banach spaces, $a>0$, $B_{a}(z_{0})=\{z\in X:\|z-z_{0}\|\leq a\}$. Suppose that $F$ is a $C^{1}$ map of $B_{a}(z_{0})$ into $Y$, with $F'(z_{0})$ invertible, and satisfying, for some $0<\theta<1$,
\begin{align}
&\|F'(z_{0})^{-1}F(z_{0})\|\leq (1-\theta)a,\label{liutan}\\
&\|F'(z_{0})\|^{-1}\|F'(z)-F'(z_{0})\|\leq\theta,~~~~\forall~z\in B_{a}.\label{liuyuxi}
\end{align}
Then there is a unique solution in $B_{a}(z_{0})$ of $F(z)=0$.
\end{lem}

After solving (\ref{F1def})-(\ref{F3def}), we will then in the next section solve the equations:
\begin{align}
&\frac{\partial}{\partial t_{1}}[J(t,\alpha(t),v(t))]=0,\label{pianT}\\
&\frac{\partial}{\partial t_{2}}[J(t,\alpha(t),v(t))]=0.\label{pianT'}
\end{align}

After solving (\ref{pianT})-(\ref{pianT'}), we have already obtained $u=\sum\limits_{i=1}^{2}\alpha_{i}(t)\delta_{P_{i},t_{i}}+v(t)$ satisfying (\ref{liuxiyij}). Then a standard argument yields that $u$ is positive, which implies that $u$ is a solution of (\ref{main equ}). Indeed, multiplying (\ref{liuxiyij}) with $u^{-}=\max\{0,-u\}$ and integrating on $\SSphere^{n}$, we have that
\begin{equation}\label{chenqi}
\|u^{-}\|^{2}=c(n)\int_{\SSphere^{n}}K(u^{-})^{\frac{2n}{n-2}}.
\end{equation}
The Sobolev embedding theorem also gives that
\begin{equation}\label{liuqi}
\int_{\SSphere^{n}}K(u^{-})^{\frac{2n}{n-2}}\leq C(n,K)\|u^{-}\|^{\frac{2n}{n-2}}.
\end{equation}
Combining (\ref{chenqi}) and (\ref{liuqi}) together, we obtain that either $u^{-}\equiv0$ or $\|u^{-}\|$ is far from zero. However, since $|u^{-}|\leq |v|$ and $v$ goes to zero in $H^{1}_{r}(\SSphere^{n})$, we obtain that $u^{-}\equiv0$. By the strong maximum principle, we can conclude that $u>0$ on $\SSphere^{n}$.

\subsubsection{A spectrum property of the linearized operator}

The following lemma is from \cite{B}, which establishes a spectrum property of the linearized operator.

\begin{lem}\label{linear}
There exists a positive constant $C=C(n)$ depending only on $n$ such that for any $v\in E_{t}$,
\begin{equation}\label{linear-}
\|v\|^{2}-c(n)\frac{n+2}{n-2}R_{0}\sum\limits_{l=1}^{2}\int_{\SSphere^{n}}\delta_{P_{l},t_{l}}^{\frac{4}{n-2}}v^{2}\geq C(n)\|v\|^{2}.
\end{equation}
\end{lem}

The following lemma can be derived from Lemma \ref{linear}.

\begin{lem}\label{linear'}
There exists a positive constant $C=C(n,K)$ depending on $n$ and $K$ such that for any $t_{1}$, $t_{2}>C$, we have that
\begin{equation}\label{dingyuan}
\|v\|^{2}-c(n)\frac{n+2}{n-2}\int_{\SSphere^{n}}K(\sum\limits_{l=1}^{2}\bar{\alpha}_{l}\delta_{P_{l},t_{l}})^{\frac{4}{n-2}}v^{2}\geq C(n,K)\|v\|^{2},~~~~\forall~v\in E_{t}.
\end{equation}
\end{lem}

\begin{proof}
By a direct computation,
\begin{align*}
&\quad\|v\|^{2}-c(n)\frac{n+2}{n-2}\int_{\SSphere^{n}}K\Big(\sum\limits_{l=1}^{2}\bar{\alpha}_{l}\delta_{P_{l},t_{l}}\Big)^{\frac{4}{n-2}}v^{2}\\
&=\left(\|v\|^{2}-c(n)\frac{n+2}{n-2}R_{0}\sum\limits_{l=1}^{2}\int_{\SSphere^{n}}\delta_{P_{l},t_{l}}^{\frac{4}{n-2}}v^{2}\right)\\
&\quad\quad-c(n)\frac{n+2}{n-2}\int_{\SSphere^{n}}K\Big[\Big(\sum\limits_{l=1}^{2}\bar{\alpha}_{l}\delta_{P_{l},t_{l}}\Big)^{\frac{4}{n-2}}-\sum\limits_{l=1}^{2}\bar{\alpha}_{l}^{\frac{4}{n-2}}\delta_{P_{l},t_{l}}^{\frac{4}{n-2}}\Big]v^{2}\\
&\quad\quad\quad-c(n)\frac{n+2}{n-2}\sum\limits_{l=1}^{2}\bar{\alpha}_{l}^{\frac{4}{n-2}}\int_{\SSphere^{n}}(K-K_{l})\delta_{P_{l},t_{l}}^{\frac{4}{n-2}}v^{2}.
\end{align*}
Then we can obtain (\ref{dingyuan}) by inserting (\ref{linear-}) in Lemma \ref{linear}, (\ref{V1-1}) in Lemma \ref{V1} and (\ref{KE1-1}) in Lemma \ref{KE1} into the above equality.

\end{proof}

\subsubsection{Solving for $(\alpha,v)$ in terms of $t$}

Now we apply the implicit theorem Lemma \ref{ift} to solve the equation
\begin{equation}\label{preequ}
F(\alpha,v)=0,
\end{equation}
where $F$ is defined as in (\ref{Fdef}).

\begin{prop}\label{infinite}
There exists a constant $C=C(n,K)>0$ depending on $n$ and $K$ such that for any $t_{1}$, $t_{2}>C$, there exists a unique $(\alpha(t),v(t))$ that solves (\ref{preequ}) and we have
\begin{equation}\label{yujia}
|\alpha(t)-\bar{\alpha}|+\|v(t)\|\leq C(n,K)\begin{cases}
(t_{1}t_{2})^{-\frac{n-2}{2}}+\sum\limits_{l=1}^{2}t_{l}^{-(n-2)},&n=4,5,\\
(t_{1}t_{2})^{-2}(\ln(t_{1}t_{2}))^{\frac{2}{3}}+\sum\limits_{l=1}^{2}t_{l}^{-4}(\ln t_{l})^{\frac{2}{3}},&n=6,\\
(t_{1}t_{2})^{-\frac{n+2}{4}}(\ln(t_{1}t_{2}))^{\frac{n+2}{2n}}+\sum\limits_{l=1}^{2}t_{l}^{-\frac{n+2}{2}},&n\geq 7.\end{cases}
\end{equation}
Moreover, $\alpha$ and $v$ are $C^{1}$ in $t$.
\end{prop}

\begin{proof}

We apply Lemma \ref{ift} to $z_{0}=(\bar{\alpha},0)$ and $F$ in the Banach space $\RR^{2}\times E_{t}$, with norm
\begin{equation*}
\|(\alpha,v)\|:=\sum\limits_{i=1}^{2}|\alpha_{i}|+\|v\|.
\end{equation*}
We prove the conclusion in three steps.

\bigskip
\noindent
\underline{Step 1.} There exists a constant $C=C(n,K)>0$ depending on $n$ and $K$ such that if $t_{1}$, $t_{2}>C$, we have that
\begin{equation}\label{kaoshiyuan}
\|F(z_{0})\|\leq C(n,K)\begin{cases}
(t_{1}t_{2})^{-\frac{n-2}{2}}+\sum\limits_{l=1}^{2}t_{l}^{-(n-2)},&n=4,5,\\
(t_{1}t_{2})^{-2}(\ln(t_{1}t_{2}))^{\frac{2}{3}}+\sum\limits_{l=1}^{2}t_{l}^{-4}(\ln t_{l})^{\frac{2}{3}},&n=6,\\
(t_{1}t_{2})^{-\frac{n+2}{4}}(\ln(t_{1}t_{2}))^{\frac{n+2}{2n}}+\sum\limits_{l=1}^{2}t_{l}^{-\frac{n+2}{2}},&n\geq 7.\end{cases}
\end{equation}
Indeed, for $i$, $j=1$, $2$, $i\neq j$,
\begin{equation}\label{wangchongwen}
F_{i}(z_{0})=c_{0}\bar{\alpha}_{i}+\bar{\alpha}_{j}\langle\delta_{P_{i},t_{i}},\delta_{P_{j},t_{j}}\rangle-c(n)\int_{\SSphere^{n}}K(\sum\limits_{l=1}^{2}\bar{\alpha}_{l}\delta_{P_{l},t_{l}})^{\frac{n+2}{n-2}}\delta_{P_{i},t_{i}},
\end{equation}
and
\begin{equation}\label{sundewei----}
F_{3}(z_{0})[\varphi]=-c(n)\int_{\SSphere^{n}}K(\sum\limits_{l=1}^{2}\bar{\alpha}_{l}\delta_{P_{l},t_{l}})^{\frac{n+2}{n-2}}\varphi,~~~~\forall~\varphi\in E_{t}.
\end{equation}
Inserting (\ref{AB1-}) in Lemma \ref{A1}, (\ref{A2-1}) in Lemma \ref{A2}, (\ref{hongbing}) in Lemma \ref{A17} into (\ref{wangchongwen}) and inserting (\ref{niuzhen}) in Lemma \ref{KE2} into (\ref{sundewei----}), we obtain (\ref{kaoshiyuan}).

\bigskip
\noindent
\underline{Step 2.} There exists a constant $C=C(n,K)>0$ depending on $n$ and $K$ such that if $t_{1}$, $t_{2}>C$, $F'$ is invertible at $z_{0}$ and
\begin{equation}\label{shipengfei}
\|F'(z_{0})^{-1}\|\leq C.
\end{equation}
Indeed, set $F'(z_{0})(\tilde{\alpha},\tilde{v})=\beta$. Direct computations yield that for $i$, $j=1$, $2$, $i\neq j$,
\begin{align}
\beta_{i}&=\left(c_{0}-c(n)\frac{n+2}{n-2}\int_{\SSphere^{n}}K(\sum\limits_{l=1}^{2}\bar{\alpha}_{l}\delta_{P_{l},t_{l}})^{\frac{4}{n-2}}\delta_{P_{i},t_{i}}^{2}\right)\tilde{\alpha}_{i}\nonumber\\
&\quad\quad+\left(\langle\delta_{P_{i},t_{i}},\delta_{P_{j},t_{j}}\rangle-c(n)\frac{n+2}{n-2}\int_{\SSphere^{n}}K(\sum\limits_{l=1}^{2}\bar{\alpha}_{l}\delta_{P_{l},t_{l}})^{\frac{4}{n-2}}\delta_{P_{1},t_{1}}\delta_{P_{2},t_{2}}\right)\tilde{\alpha}_{j}\nonumber\\
&\quad\quad-c(n)\frac{n+2}{n-2}\int_{\SSphere^{n}}K(\sum_{l = 1}^2\bar{\alpha}_{l}\delta_{P_{l},t_{l}})^{\frac{4}{n-2}}\delta_{P_{i},t_{i}}\tilde{v},\label{liufangkunhan}
\end{align}
and for any $\varphi\in E_{t}$,
\begin{align*}
\beta_{3}[\varphi]&=-\frac{n+2}{n-2}c(n)\int_{\SSphere^{n}}K(\sum\limits_{l=1}^{2}\bar{\alpha}_{l}\delta_{P_{l},t_{l}})^{\frac{4}{n-2}}(\sum\limits_{q=1}^{2}\tilde{\alpha}_{q}\delta_{P_{q},t_{q}})\varphi\\
&\quad\quad\quad\quad+\langle \tilde{v},\varphi\rangle-\frac{n+2}{n-2}c(n)\int_{\SSphere^{n}}K(\sum\limits_{l=1}^{2}\bar{\alpha}_{l}\delta_{P_{l},t_{l}})^{\frac{4}{n-2}}\tilde{v}\varphi.
\end{align*}

Inserting (\ref{V1-2}) in Lemma \ref{V1}, (\ref{KE1-2}) in Lemma \ref{KE1} and (\ref{dingyuan}) in Lemma \ref{linear'} into the above equality with $\varphi=\tilde{v}$, we have that for $t_{1}$ and $t_{2}$ sufficiently large,
\begin{equation}\label{zhaozhicheng}
\|\tilde{v}\|\leq C(n,K)(\|\beta_{3}\|+o(1)|\tilde{\alpha}|).
\end{equation}

Inserting (\ref{hongbing-}) and (\ref{hongbing--}) in Lemma \ref{A17}, (\ref{A2-1}) in Lemma \ref{A2}, (\ref{V1-2}) in Lemma \ref{V1}, (\ref{KE1-2}) in Lemma \ref{KE1} into (\ref{liufangkunhan}), we have that for $t_{1}$ and $t_{2}$ sufficiently large,
\begin{equation*}
|\tilde{\alpha}_{i}|\leq C(n,K)(|\beta_{i}|+o(1)(|\tilde{\alpha}_{j}|+\|\tilde{v}\|)).
\end{equation*}

Inserting (\ref{zhaozhicheng}) into the above inequality, we have that
\begin{equation}\label{niumiaomiao}
|\tilde{\alpha}_{i}|\leq C(n,K)(|\beta_{i}|+\|\beta_{3}\|+o(1)|\tilde{\alpha}_{j}|).
\end{equation}
Analogously, we also have that
\begin{equation}\label{songqianqian}
|\tilde{\alpha}_{j}|\leq C(n,K)(|\beta_{j}|+\|\beta_{3}\|+o(1)|\tilde{\alpha}_{i}|).
\end{equation}
Combining (\ref{zhaozhicheng}), (\ref{niumiaomiao}) and (\ref{songqianqian}) together, we can obtain that $\|(\tilde{\alpha},\tilde{v})\|\leq C(n,K)\|\beta\|$.

\bigskip
\noindent
\underline{Step 3.} We choose suitable $\theta$ and $a$ to show that there exits a constant $C=C(n,K)>0$ depending on $n$ and $K$ such that for $t_{1}$, $t_{2}>C$, we have that (\ref{liutan}) and (\ref{liuyuxi}). Indeed, direct computations yield that for $i$, $j=1$, $2$, $i\neq j$,
\begin{align}
&(F_{i}'(z)-F_{i}'(z_{0}))(\tilde{\alpha},\tilde{v})\nonumber\\
&\qquad =-c(n)\frac{n+2}{n-2}\Big[\tilde{\alpha}_{i}\int_{\SSphere^{n}}K\Big(|u|^{\frac{4}{n-2}}-(\sum\limits_{l=1}^{2}\bar{\alpha}_{l}\delta_{P_{l},t_{l}})^{\frac{4}{n-2}}\Big)\delta_{P_{i},t_{i}}^{2}\nonumber\\
&\qquad\qquad +\tilde{\alpha}_{j}\int_{\SSphere^{n}}K\Big(|u|^{\frac{4}{n-2}}-(\sum\limits_{l=1}^{2}\bar{\alpha}_{l}\delta_{P_{l},t_{l}})^{\frac{4}{n-2}}\Big)\delta_{P_{i},t_{i}}\delta_{P_{j},t_{j}}\nonumber\\
&\qquad\qquad+\int_{\SSphere^{n}}K\Big(|u|^{\frac{4}{n-2}}-(\sum\limits_{l=1}^{2}\bar{\alpha}_{l}\delta_{P_{l},t_{l}})^{\frac{4}{n-2}}\Big)\delta_{P_{i},t_{i}}\tilde{v}\Big],\label{liufangkunhan--}
\end{align}
and for any $\varphi\in E_{t}$,
\begin{align}
&(F_{3}'(z)-F_{3}'(z_{0}))(\tilde{\alpha},\tilde{v})[\varphi]\nonumber\\
&\qquad=-c(n)\frac{n+2}{n-2}\Big[\tilde{\alpha}_{i}\int_{\SSphere^{n}}K\Big(|u|^{\frac{4}{n-2}}-(\sum\limits_{l=1}^{2}\bar{\alpha}_{l}\delta_{P_{l},t_{l}})^{\frac{4}{n-2}}\Big)\delta_{P_{i},t_{i}}\varphi\nonumber\\
&\qquad\qquad+\tilde{\alpha}_{j}\int_{\SSphere^{n}}K\Big(|u|^{\frac{4}{n-2}}-(\sum\limits_{l=1}^{2}\bar{\alpha}_{l}\delta_{P_{l},t_{l}})^{\frac{4}{n-2}}\Big)\delta_{P_{j},t_{j}}\varphi\nonumber\\
&\qquad\qquad+\int_{\SSphere^{n}}K\Big(|u|^{\frac{4}{n-2}}-(\sum\limits_{l=1}^{2}\bar{\alpha}_{l}\delta_{P_{l},t_{l}})^{\frac{4}{n-2}}\Big)\tilde{v}\varphi\Big].\label{liufangkunhan--''}
\end{align}

Inserting (\ref{yuanmeng}), (\ref{yangxiu}), (\ref{ranshiyou}) in Lemma \ref{lidehuang} into (\ref{liufangkunhan--}) and inserting (\ref{ranshiyou}), (\ref{zhupinxiu}) in Lemma \ref{lidehuang} into (\ref{liufangkunhan--''}), we can conclude that there exists a constant $C=C(n)>0$ depending only on $n$ such that for $t_{1}$, $t_{2}>C$ and $|\alpha_{1}-\bar{\alpha}_{1}|\leq\frac{1}{2}\bar{\alpha}_{1}$, $|\alpha_{2}-\bar{\alpha}_{2}|\leq\frac{1}{2}\bar{\alpha}_{2}$,
\begin{equation*}
\|F'(z)-F'(z_{0})\|\leq C(n,K)(\|v\|^{\frac{4}{n-2}}+\|v\|+|\alpha-\bar{\alpha}|^{\frac{4}{n-2}}+|\alpha-\bar{\alpha}|).
\end{equation*}
By the above inequality and (\ref{shipengfei}), we have that
\begin{equation*}
\|F'(z_{0})\|^{-1}\|F'(z)-F'(z_{0})\|\leq C(n,K)(a^{\frac{4}{n-2}}+a),~~~~\forall~|z-z_{0}|\leq a.
\end{equation*}
By (\ref{kaoshiyuan}) and (\ref{shipengfei}), we also have that
\begin{equation*}
\|F'(z_{0})^{-1}F(z_{0})\|\leq C(n,K)\begin{cases}
(t_{1}t_{2})^{-\frac{n-2}{2}}+\sum\limits_{l=1}^{2}t_{l}^{-(n-2)},&n=4,5,\\
(t_{1}t_{2})^{-2}(\ln(t_{1}t_{2}))^{\frac{2}{3}}+\sum\limits_{l=1}^{2}t_{l}^{-4}(\ln t_{l})^{\frac{2}{3}},&n=6,\\
(t_{1}t_{2})^{-\frac{n+2}{4}}(\ln(t_{1}t_{2}))^{\frac{n+2}{2n}}+\sum\limits_{l=1}^{2}t_{l}^{-\frac{n+2}{2}},&n\geq 7.\end{cases}
\end{equation*}
Now we fix $\theta=\frac{1}{2}$, take $\tilde{C}>2C(n,K)$ and set
$$a=\tilde{C}\begin{cases}
(t_{1}t_{2})^{-\frac{n-2}{2}}+\sum\limits_{l=1}^{2}t_{l}^{-(n-2)},&n=4,5,\\
(t_{1}t_{2})^{-2}(\ln(t_{1}t_{2}))^{\frac{2}{3}}+\sum\limits_{l=1}^{2}t_{l}^{-4}(\ln t_{l})^{\frac{2}{3}},&n=6,\\
(t_{1}t_{2})^{-\frac{n+2}{4}}(\ln(t_{1}t_{2}))^{\frac{n+2}{2n}}+\sum\limits_{l=1}^{2}t_{l}^{-\frac{n+2}{2}},&n\geq 7.\end{cases}$$
Then we can conclude that (\ref{liutan})-(\ref{liuyuxi}) hold for $t_{1}$, $t_{2}$ sufficiently large.

We can now apply Lemma \ref{ift} to obtain the existence and uniqueness of the solution ($\alpha(t)$, $v(t)$) of (\ref{preequ}). Furthermore, $\alpha$, $v$ are $C^{1}$ in $t$.
\end{proof}

\subsection{Solvability of the reduced finite-dimensional problem}\label{SSec:Sol}

In this section, we solve the system:
\begin{equation}\label{sysH}
H(t):=(H_{1}(t),H_{2}(t))=(0,0),
\end{equation}
where
\begin{equation*}
H_{i}(t):=\frac{\partial}{\partial t_{i}}[J(t,\alpha(t),v(t))],~~~~i=1,2
\end{equation*}
with $(\alpha(t),v(t))$ obtained from Proposition \ref{infinite}.

Using the equation (\ref{F1def})-(\ref{F3def}) for $(\alpha(t),v(t))$ and the fact that $v(t)\in E_{t}$, a direct computation gives for $i$, $j=1$, $2$, $i\neq j$ that
\begin{align*}
H_{i}(t)&=\alpha_{i}\left(\alpha_{j}\frac{\partial}{\partial t_{i}}\langle \delta_{P_{i},t_{i}},\delta_{P_{j},t_{j}}\rangle-c(n)\int_{\SSphere^{n}}K|u|^{\frac{4}{n-2}}u\frac{\partial\delta_{P_{i},t_{i}}}{\partial t_{i}}\right)\\
&\quad\quad\quad\quad-c(n)\int_{\SSphere^{n}}K|u|^{\frac{4}{n-2}}u(1-\Pi_{t})(\frac{\partial v}{\partial t_{i}}).
\end{align*}
In order to estimate $H_{i}(t)$, we split $H_{i}(t)$ into two parts:
\begin{equation*}
H_{i}(t)=\bar{H}_{i}(t)+(H_{i}(t)-\bar{H}_{i}(t)),
\end{equation*}
where
\begin{equation*}
\bar{H}_{i}(t):=\bar{\alpha}_{i}\left(\bar{\alpha}_{j}\frac{\partial}{\partial t_{i}}\langle \delta_{P_{i},t_{i}},\delta_{P_{j},t_{j}}\rangle-c(n)\int_{\SSphere^{n}}K(\sum\limits_{l=1}^{2}\bar{\alpha}_{l}\delta_{P_{l},t_{l}})^{\frac{n+2}{n-2}}\frac{\partial\delta_{P_{i},t_{i}}}{\partial t_{i}}\right).
\end{equation*}

The following proposition gives the estimate of $\bar{H}_{i}(t)$.
\begin{prop}\label{youyue}
There exists a constant $C=C(n,K)>0$ depending on $n$ and $K$ such that if $t_{1}$, $t_{2}>C$, we have that
\begin{equation}\label{mocha}
\bar{H}_{i}(t)=t_{i}^{-1}[A_{i}t_{i}^{-(n-2)}+B(t_{1}t_{2})^{-\frac{n-2}{2}}+D_{i}t_{i}^{-\mu_{i}}+\tilde{R}_{i}(t)],
\end{equation}
where
\begin{align}
A_{i}&=\frac{2^{2n-6 }(n-2)^3}{n-1} R_0^{\frac{n-2}{2}}|\SSphere^{n-1}| a_{i}K^{-\frac{n}{2}}(P_{i}),\label{defA}
\\
B&=2^{n-3} (n-2)^2R_0^{\frac{n-2}{2}}|\SSphere^{n-1}| (K(P_{1})K(P_{2}))^{-\frac{n-2}{4}},\label{defB}
\\
D_{i}&=2^{n-3+\mu_i}\mu_i R_0^{\frac{n-2}{2}}|\SSphere^{n-1}|\Big(\int_{0}^{\infty}\frac{r^{n+\mu_i-1}dr}{(1+r^{2})^{n}}\Big)b_{i}K^{-\frac{n}{2}}(P_{i}),\label{defD}
\end{align}
and
\begin{equation*}
\tilde{R}_{i}(t)=o(t_{i}^{-\mu_{i}})+(t_{1}t_{2})^{-\frac{n-2}{2}}\begin{cases}
O(t_{1}^{-2}\ln t_{1}+t_{2}^{-2}\ln t_{2}+(t_{1}t_{2})^{-1}\ln(t_{1}t_{2})),&n=4,\\
O(t_{1}^{-2}+t_{2}^{-2}+(t_{1}t_{2})^{-1}\ln(t_{1}t_{2})),&n\geq5.
\end{cases}
\end{equation*}

\end{prop}

\begin{proof}
A direct computation gives that
\begin{align*}
\bar{H}_{i}(t)
&=-c(n)\bar{\alpha}_{i}^{\frac{2n}{n-2}}\int_{\SSphere^{n}}K\delta_{P_{i},t_{i}}^{\frac{n+2}{n-2}}\frac{\partial \delta_{P_{i},t_{i}}}{\partial t_{i}}+\bar{\alpha}_{i}\bar{\alpha}_{j}\frac{\partial}{\partial t_{i}}\langle \delta_{P_{i},t_{i}},\delta_{P_{j},t_{j}}\rangle\\
&\qquad-c(n)\bar{\alpha}_{i}\bar{\alpha}_{j}^{\frac{n+2}{n-2}}\int_{\SSphere^{n}}K\delta_{P_{j},t_{j}}^{\frac{n+2}{n-2}}\frac{\partial \delta_{P_{i},t_{i}}}{\partial t_{i}}-\frac{n+2}{n-2}c(n)\bar{\alpha}_{i}^{\frac{n+2}{n-2}}\bar{\alpha}_{j}\int_{\SSphere^{n}}K\delta_{P_{i},t_{i}}^{\frac{4}{n-2}}\delta_{P_{j},t_{j}}\frac{\partial \delta_{P_{i},t_{i}}}{\partial t_{i}}\\
&\qquad -c(n)\bar{\alpha}_{i}\int_{\SSphere^{n}}K\Big((\sum\limits_{l=1}^{2}\bar{\alpha}_{l}\delta_{P_{l},t_{l}})^{\frac{n+2}{n-2}}-\sum\limits_{l=1}^{2}(\bar{\alpha}_{l}\delta_{P_{l},t_{l}})^{\frac{n+2}{n-2}}\\
    &\qquad\qquad\qquad -\frac{n+2}{n-2}(\bar{\alpha}_{i}\delta_{P_{i},t_{i}})^{\frac{4}{n-2}}(\bar{\alpha}_{j}\delta_{P_{j},t_{j}})\Big)\frac{\partial \delta_{P_{i},t_{i}}}{\partial t_{i}}.
\end{align*}
Then (\ref{mocha}) follows from the above equality, (\ref{A2-2}) in Lemma \ref{A2}, (\ref{A3-1}) in Lemma \ref{A3}, (\ref{A4-1}) and (\ref{A4-2}) in Lemma \ref{A4}, (\ref{A16-1}) in Lemma \ref{A16}.
\end{proof}

Let
\begin{equation*}
\hat{R}_{i}(t):=t_{i}(H_{i}(t)-\bar{H}_{i}(t)).
\end{equation*}
The following proposition gives the estimate of $\hat{R}_{i}(t)$.
\begin{prop}\label{youyue--}
There exists a constant $C=C(n,K)>0$ depending on $n$ and $K$ such that if $t_{1}$, $t_{2}>C$, we have that
\begin{equation}\label{mocha-}
|\hat{R}_{i}(t)|\leq C(n,K)\begin{cases}
(t_{1}t_{2})^{-(n-2)}+\sum\limits_{l=1}^{2}t_{l}^{-2(n-2)},&n=4,5,\\
(t_{1}t_{2})^{-4}(\ln(t_{1}t_{2}))^{\frac{4}{3}}+\sum\limits_{l=1}^{2}t_{l}^{-8}(\ln t_{l})^{\frac{4}{3}},&n=6,\\
(t_{1}t_{2})^{-\frac{n+2}{2}}(\ln(t_{1}t_{2}))^{\frac{n+2}{n}}+\sum\limits_{l=1}^{2}t_{l}^{-(n+2)},&n\geq 7.
\end{cases}
\end{equation}
\end{prop}

\begin{proof}
A direct computation gives that
\begin{align*}
H_{i}(t)-\bar{H}_{i}(t)
&=(\alpha_{i}\alpha_{j}-\bar{\alpha}_{i}\bar{\alpha}_{j})\frac{\partial}{\partial t_{i}}\langle \delta_{P_{i},t_{i}},\delta_{P_{j},t_{j}}\rangle-c(n)(\alpha_{i}-\bar{\alpha}_{i})\int_{\SSphere^{n}}K|u|^{\frac{4}{n-2}}u\frac{\partial \delta_{P_{i},t_{i}}}{\partial t_{i}}\\
&\qquad-c(n)\bar{\alpha}_{i}\int_{\SSphere^{n}}K\Big[|u|^{\frac{4}{n-2}}u-\Big(\sum\limits_{l=1}^{2}\bar{\alpha}_{l}\delta_{P_{l},t_{l}}\Big)^{\frac{n+2}{n-2}}\Big]\frac{\partial \delta_{P_{i},t_{i}}}{\partial t_{i}}\\
&\qquad-c(n)\int_{\SSphere^{n}}K|u|^{\frac{4}{n-2}}u(1-\Pi_{t})(\frac{\partial v}{\partial t_{i}}).
\end{align*}
Then (\ref{mocha-}) follows from the above identity, (\ref{yujia}) in Proposition \ref{infinite}, (\ref{A2-2}) in Lemma \ref{A2}, (\ref{linyang}), (\ref{helin}), (\ref{lunyang}) in Lemma \ref{liutengfei}, (\ref{zhongchangjiang}), (\ref{high}) in Lemma \ref{zhognshizhang}.
\end{proof}

Now we can prove our main theorem.

\begin{proof}[Proof of Theorems \ref{main thm}]
We only give the proof of part (a). That of part (b) is analogous.

Let $\nu=(K_{1},K_{2},-a_{1},-a_{2})\in\RR^{4}_{+} \setminus \Sigma^0$ be a point near $\bar \nu \in \Sigma^0$ and consider the system \eqref{sysH}. By Propositions \ref{youyue} and \ref{youyue--}, system (\ref{sysH}) can be equivalently written as
\begin{eqnarray}\label{fangchengzu2}
\left(\begin{array}{cc}
A_{1}&B\\
B&A_{2}
\end{array}\right)\left(\begin{array}{c}
t_{2}^{\frac{n-2}{2}}\\
t_{1}^{\frac{n-2}{2}}
\end{array}\right)=-\left(\begin{array}{c}
t_{1}^{-(\mu_{1}-(n-2))}t_{2}^{\frac{n-2}{2}}(D_{1}+t_{1}^{\mu_{1}}R_{1}(t))\\
t_{2}^{-(\mu_{2}-(n-2))}t_{1}^{\frac{n-2}{2}}(D_{2}+t_{2}^{\mu_{2}}R_{2}(t))
\end{array}\right),
\end{eqnarray}
where $R_{i}(t):=\tilde{R}_{i}(t)+\hat{R}_{i}(t)$ and $A_{i}$, $B$, $D_{i}$ are defined as in (\ref{defA})-(\ref{defD}) for $i=1$, $2$. Note that since $\nu \notin \Sigma^0$,
\[
\delta = \delta_\nu := A_1 A_2 - B^2 \neq 0.
\]
Moreover, $\delta > 0$ if $\nu \in \Sigma^+$ and $\delta < 0$ if $\nu \in \Sigma^-$.

By exchanging the role of the north and the south pole, we may assume without loss of generality that $\mu_1 \leq \mu_2$. We now make a change of variables $\beta = \delta^{-1} t_1^{-(\mu_1 - (n-2))}$, $s = (t_{2}/t_{1})^{(n-2)/2}$ and rewrite \eqref{fangchengzu2} in terms of $\beta$ and $s$. First, we rewrite \eqref{fangchengzu2} in terms of $s$ and $t_1$:
\[
\begin{pmatrix}A_1 & B\\B & A_2\end{pmatrix}
\begin{pmatrix}s\\1\end{pmatrix}
=\begin{pmatrix}t_{1}^{-(\mu_{1}-(n-2))}s(D_{1}+ \phi_1(t_{1},s)\\
s^{-\frac{2(\mu_{2}-(n-2))}{n-2}}t_{1}^{-(\mu_{2}-(n-2))}(D_{2}+t_{2}^{\mu_{2}}R_{2}(t))\end{pmatrix},
\]
where $\phi_i(t_1,s) = t_i^{\mu_i}R_i(t_1,s^{2/(n-2)}t_1)$. Performing Gaussian elimination and noting that $\delta =A_{1}A_{2}-B^{2}\neq0$, we obtain
\begin{align}\label{zouxiyu}
&\left(\begin{array}{cc}
A_{1}&B\\
0&1
\end{array}\right)\left(\begin{array}{c}
s\\
1
\end{array}\right)\nonumber\\
&\quad=-\left(\begin{array}{c}
\delta \beta  s (D_{1}+\hat \phi_{1}(\delta,\beta,s))\\
-B \beta s(D_{1}+ \hat \phi_{1}(\delta,\beta,s))+A_{1}s^{-\frac{2(\mu_{2}-(n-2))}{n-2}} \beta (\delta \beta)^{\frac{\mu_2 - \mu_1}{\mu_1 - (n-2)}}(D_{2}+\hat \phi_{2}(\delta,\beta,s))
\end{array}\right),
\end{align}
where $\hat \phi_i(\delta,\beta,s) = \phi_i((\delta \beta)^{-\frac{1}{\mu_1 - (n-2)}},s)$. This is more conveniently written as
\begin{equation}\label{baorui}
\Phi_{\nu}(\beta,s):=(\Phi_{\nu}^{1}(\beta,s),\Phi_{\nu}^{2}(\beta,s))=(0,0),
\end{equation}
where
\begin{align*}
\Phi_{\nu}^{1}(\beta,s)
    &:=A_{1}s+B+\beta s\delta(D_{1}+\hat{\phi}_{1}(\delta,\beta,s)),\\
\Phi_{\nu}^{2}(\beta,s)
    &:=1-B\beta s\left((D_{1}+\hat{\phi}_{1}(\delta,\beta,s))-\frac{A_{1}}{B}s^{-\frac{2\mu_2}{n-2}+1} (\delta \beta)^{\frac{\mu_2 - \mu_1}{\mu_1 - (n-2)}}(D_{2}+\hat{\phi}_{2}(\delta,\beta,s))\right).
\end{align*}
Note that, by the estimates for $\tilde R_i$ and $\hat R_i$ in Propositions \ref{youyue} and \ref{youyue--}, whenever $\beta, s \leq M$ for some fixed $M$, 
\[
|\hat\phi_i(\delta,\beta,s)| \le o(1) \text{ as } \nu \rightarrow \bar \nu.
\]

We proceed depending on whether $\mu_1 = \mu_2$ or $\mu_1 < \mu_2$.

\medskip
\noindent
\underline{Case 1.} $\mu_{1}=\mu_{2}$, where we also have by hypothesis that $\bar \nu \in S_{\mu,b}^+$ (that is, $F_{\mu,b}(\bar \nu) > 0$).

We need to show that, when $\nu \in \Sigma^-$ (resp., $\nu \in \Sigma^+$) is sufficiently close to $\bar \nu$, \eqref{baorui} has a solution $(\beta,s)$ with $\beta < 0$ (resp., no solution $(\beta,s)$ with $\beta > 0$). (Recall that $\nu \in \Sigma^-$ is equivalent to $\delta < 0$, and $\nu \in \Sigma^+$ is equivalent to $\delta > 0$.)

Note that, in view of the fact that $F_{\mu,b}(\bar \nu) \neq 0$, the equation
\begin{equation}\label{kunming}
\Phi_{\bar{\nu}}(\bar{\beta},\bar{s})=(0,0)
\end{equation}
has a unique solution given by
\begin{equation*}
(\bar{\beta},\bar{s})=\Big((-\frac{\bar{A}_{1}}{\bar{B}^{2}})(\bar{D}_{1}+\bar{D_{2}}(-\frac{\bar{A}_{1}}{\bar{B}})^{\frac{2\mu_1}{n-2}})^{-1},-\frac{\bar{B}}{\bar{A}_{1}}\Big).
\end{equation*}
Moreover, since $F_{\mu,b}(\bar \nu) > 0$, $\bar \beta < 0$. A direct computation gives that the Jacobian of $\Phi_{\bar \nu}$ at $(\bar\beta,\bar s)$ is 
\begin{align*}
\bar{M}=\left(\begin{array}{cc}
0 &\bar{A}_{1}\\
-\bar{B}\bar{s}(\bar{D}_{1}+\bar{D_{2}}\bar{s}^{-\frac{2\mu_1}{n-2}})& -\bar{B}\bar{\beta}(\bar{D}_{1}-2(\frac{\mu_1}{n-2}-1)\bar{D_{2}}\bar{s}^{-\frac{2\mu_1}{n-2}})
\end{array}\right),
\end{align*}
and the invertibility of this matrix is equivalent to $F_{\mu,b}(\bar \nu) \neq 0$. If $\Phi_\nu$ was a $C^1$ perturbation of $\Phi_{\bar \nu}$, we could at this point invoke the implicit function theorem to obtain the existence of a solution to \eqref{baorui}. Since our estimates for the perturbation terms $\hat \phi_i$ are $C^0$ estimates, we argue via the Brouwer fixed point theorem as follows. For $|\nu-\bar{\nu}|\leq |\nu|/2$, by inserting (\ref{kunming}) into (\ref{baorui}), we can rewrite (\ref{baorui}) as \begin{equation}\label{mengsuxia}
\begin{pmatrix}
\beta\\s 
\end{pmatrix}= \begin{pmatrix}\bar\beta \\\bar s\end{pmatrix} + \bar M^{-1}\begin{pmatrix}
O(|s-\bar{s}|^{2})+ \textrm{error}\\
O(|s-\bar{s}|^{2}+|\beta-\bar{\beta}|^{2})+\textrm{error}
\end{pmatrix} := T_\nu\left(\begin{pmatrix}
\beta\\s 
\end{pmatrix}\right),
\end{equation}
where the error terms satisfy
\[
|\textrm{error}| \leq \varrho(|\nu - \bar \nu|)
\]
for some non-negative continuous increasing function $\varrho$ with $\varrho(0) = 0$. Thus, by choosing a constant $C>0$ sufficiently large, when $|\nu-\bar{\nu}|$ sufficiently small, the map $T_\nu$ in (\ref{mengsuxia}) defines a continuous map from $\bar B_{C\varrho(|\nu - \bar \nu|)}(\bar{\beta},\bar{s})$ to itself. By the Brouwer fixed point theorem, \eqref{baorui} has a solution in $\bar B_{C\varrho(|\nu - \bar \nu|)}(\bar{\beta},\bar{s})$. Since $\bar \beta < 0$, we have, after shrinking $|\nu - \bar \nu|$ if necessary, that $\beta < 0$. This gives the desired two-bump solution.

For the other part, suppose by contradiction that there exists a sequence $\{\nu_m\} \subset \Sigma^+$ converging to $\bar \nu$ such that equation (\ref{main equ}) with $K=K_{\nu_m}$ admits a positive solution $u_{m}$ satisfying $\max\limits_{\SSphere^{n}}u_{m}\rightarrow\infty$ as $m\rightarrow\infty$. Then $\{u_{m}\}$ must blow up at the north pole and the south pole. Consequently, equation \eqref{baorui} with $\nu = \nu_m$ has a solution $(\beta_m,s_m)$ with $\beta_m > 0$. By inspection, the equation $\Phi_{\nu_m}^1(\beta_m,s_m) = 0$ implies $s_m \rightarrow \bar s$, and this together with $\Phi_{\nu_m}^2(\beta_m,s_m) = 0$ implies $\beta_m \rightarrow \bar \beta$, which is a contradiction as $\bar \beta < 0$. This completes the proof in Case 1.

\medskip
\noindent
\underline{Case 2.} $\mu_{1}<\mu_{2}$, where we also have by hypothesis that $b_1 > 0$.

In this case, thanks to the fact that $b_1 > 0$, the unique solution of \eqref{kunming} is given by
\begin{equation*}
(\bar{\beta},\bar{s})=\Big(-\frac{\bar{A}_{1}}{\bar{B}^{2}\bar{D}_{1}},-\frac{\bar{B}}{\bar{A}_{1}}\Big).
\end{equation*}
Again, we also have $\bar \beta < 0$. The Jacobian of $\Phi_{\bar\nu}$ at $(\bar\beta,\bar s)$ is
\begin{align*}
\bar{M}=\left(\begin{array}{cc}
0 & \bar{A}_{1}\\
-\bar{B}\bar{D}_{1}\bar{s} & -\bar{B}\bar{D}_{1}\bar{\beta}
\end{array}\right),
\end{align*}
which is always invertible. The conclusion follows from the same argument as in Case 1. We skip the details.
\end{proof}

\appendix

\section{Appendix: Technical estimates}

We collect estimates for various integral identities and estimates that were needed in the body. The proofs of these estimates are elementary in nature and done by pulling back to $\RR^n$ via the stereographic projection
\begin{equation}\label{AA1}
x^{m}=\frac{2y^{m}}{1+|y|^{2}},~~x^{n+1}=\frac{1-|y|^{2}}{1+|y|^{2}}P_{2}^{n+1},~~~~m=1,\cdots,n,
\end{equation}
where one has
\begin{align}\delta_{P_{1},t_{1}}(x(y))
    &=\left(\frac{t_{1}^{-1}(1+|y|^{2})}{1+t_{1}^{-2}|y|^{2}}\right)^{\frac{n-2}{2}},~~~~\forall~y\in\RR^{n},
\\
\frac{\partial\delta_{P_{1},t_{1}}}{\partial t_{1}}(x(y))
    &=-\frac{n-2}{2}t_{1}^{-1}\delta_{P_{1},t_{1}}\frac{1-t_{1}^{-2}|y|^{2}}{1+t_{1}^{-2}|y|^{2}},~~~~\forall~y\in\RR^{n},\\
\label{AA2}
\delta_{P_{2},t_{2}}(x(y))
    &=\left(\frac{t_{2}(1+|y|^{2})}{1+t_{2}^{2}|y|^{2}}\right)^{\frac{n-2}{2}},~~~~\forall~y\in\RR^{n},\\
\frac{\partial\delta_{P_{2},t_{2}}}{\partial t_{2}}(x(y))
    &=\frac{n-2}{2}t_{2}^{-1}\delta_{P_{2},t_{2}}\frac{1-t_{2}^{2}|y|^{2}}{1+t_{2}^{2}|y|^{2}},~~~~\forall~y\in\RR^{n}.\nonumber
\label{AA3}
\end{align}

Throughout the appendix, we assume that $K$ satisfies \eqref{higher0X}. This is equivalent to the condition that
\begin{align*}
K(r)
    &=K_{2}+2^{n-2}a_{2}r^{n-2}+2^{\mu_{2}}b_{2}r^{\mu_{2}}+\bar{R}_{2}(r),~~~~as~r:=|y|\rightarrow0^{+},\\
K(r)
    &=K_{1}+2^{n-2}a_{1}r^{-(n-2)}+2^{\mu_{1}}b_{1}r^{-\mu_{1}}+\bar{R}_{1}(r),~~~~as~r:=|y|\rightarrow \infty,
\end{align*}
with the error terms being controlled up to the first order derivatives. All implicit constants in subsequent estimates depend only on $n$, and an upper bound for $K_i$, $|a_i|$, $|b_i|$ and the function $\omega$ in \eqref{higher0X}. 

As the proofs are elementary, at places we will indicate only the key points rather than full details.

\subsection{Bubble interaction estimates}

\begin{lem}\label{A1}
\begin{equation}\label{AB1-}
c_0 = \langle\delta_{P_{i},t_{i}},\delta_{P_{i},t_{i}}\rangle=2^{n-2}n(n-2)|\SSphere^{n-1}|\int_{0}^{\infty}\frac{r^{n-1}dr}{(1+r^{2})^{n}},
\end{equation}
and
\begin{equation}\label{AB1}
\left\langle\frac{\partial\delta_{P_{i},t_{i}}}{\partial t_{i}},\frac{\partial\delta_{P_{i},t_{i}}}{\partial t_{i}}\right\rangle=\frac{2^{n-4}n(n-2)^2(n+2)}{n+1}|\SSphere^{n-1}|\int_{0}^{\infty}\frac{r^{n-1}dr}{(1+r^{2})^{n}}t_{i}^{-2}.
\end{equation}
\end{lem}

\begin{proof}
We will only consider $i = 1$. Recall that $-L_{g_0} \delta_{P,t} = \frac{n(n-2)}{4}\delta_{P,t}^{\frac{n+2}{n-2}}$ and $-L_{g_0} \frac{\partial\delta_{P,t}}{\partial t} = \frac{n(n+2)}{4}\delta_{P,t}^{\frac{4}{n-2}}\frac{\partial\delta_{P,t}}{\partial t}$. Therefore, 
\begin{align*}
\langle\delta_{P_{1},t_{1}},\delta_{P_{1},t_{1}}\rangle 
    &= \frac{n(n-2)}{4}\int_{\mathbb{S}^n} \delta_{P_1,t_1}^{\frac{2n}{n-2}}dv_{g_0}
        = 2^{n-2}n(n-2)\int_{\mathbb{R}^n} \frac{t_1^{-n}dy}{(1+ t_1^{-2} |y|^2)^n},\\
\left\langle\frac{\partial\delta_{P_{1},t_{1}}}{\partial t_{1}},\frac{\partial\delta_{P_{1},t_{1}}}{\partial t_{1}}\right\rangle
    &= \frac{n(n+2)}{4} \int_{\mathbb{S}^n} \delta_{P_1,t_1}^{\frac{4}{n-2}} \big(\frac{\partial\delta_{P_{1},t_{1}}}{\partial t_{1}}\big)^2 dv_{g_0}\\
    &= 2^{n-4} n(n-2)^2(n+2)\int_{\mathbb{R}^n} \frac{t_1^{-n-2}(1 - t_1^{-2}|y|^2)^2 dy}{(1 + t_1^{-2}|y|^2)^{n+2}}.
\end{align*}
Making the change of variables $z = t_1^{-1}y$ and then switching to spherical coordinates, \eqref{AB1-} follows immediately, while for \eqref{AB1} we need to show
\begin{equation}
\int_0^\infty \frac{(1-r^2)^2r^{n-1}}{(1+r^2)^{n+2}}dr = \frac{1}{n+1}\int_0^\infty \frac{r^{n-1}}{(1+r^2)^n}dr.
    \label{A1-P1}
\end{equation}
To see this, we use the integration-by-parts identity
\begin{equation*}
\int_0^\infty  \frac{r^{a+1}}{(1+r^2)^{b+1}}dr
    = -\frac{1}{2b} \int_0^\infty  r^a \frac{d}{dr}\Big(\frac{1}{(1+r^2)^b}\Big)dr
    = \frac{a}{2b} \int_0^\infty   \frac{r^{a-1}}{(1+r^2)^{b}}dr,
\end{equation*}
which holds for $2b > a > 0$. This implies
\begin{equation*}
\int_0^\infty  \frac{r^{a-1}}{(1+r^2)^{b+1}}dr
    = \int_0^\infty \Big(\frac{r^{a-1}}{(1+r^2)^{b}} - \frac{r^{a+1}}{(1+r^2)^{b+1}}\Big)dr
    = \frac{2b-a}{2b} \int_0^\infty   \frac{r^{a-1}}{(1+r^2)^{b}}dr.
\end{equation*}
In particular,
\[
\int_0^\infty  \frac{r^{n+1}}{(1+r^2)^{n+2}}dr
    =  \frac{n}{2(n+1)} \int_0^\infty   \frac{r^{n-1}}{(1+r^2)^{n+1}}dr
    =  \frac{n}{4(n+1)}  \int_0^\infty   \frac{r^{n-1}}{(1+r^2)^{n}}dr.
\]
The identity \eqref{A1-P1} follows from the above identiy and the algebraic identity
\[
  \frac{(1-r^2)^2r^{n-1}}{(1+r^2)^{n+2}}
    = \frac{r^{n-1}}{(1+r^2)^{n}} - \frac{4r^{n+1}}{(1 + r^2)^{n+2}}.
\]
The proof is complete
\end{proof}

\begin{lem}\label{A2}
If $t_{1}t_{2}\geq e$, we have for $i \neq j$ that
\begin{align}
\langle\delta_{P_{i},t_{i}},\delta_{P_{j},t_{j}}\rangle
    &= 2^{n-2}(n-2)|\SSphere^{n-1}|(t_{1}t_{2})^{-\frac{n-2}{2}}
        \big[1+O_{n}((t_{1}t_{2})^{-2}\ln(t_{1}t_{2}))\big],
    \label{A2-1}\\
\Big\langle\frac{\partial \delta_{P_{i},t_{i}}}{\partial t_{i}},\delta_{P_{j},t_{j}}\Big\rangle
&=- 2^{n-3} (n-2)^2|\SSphere^{n-1}|t_{i}^{-1}(t_{1}t_{2})^{-\frac{n-2}{2}}  \big[1+O_{n}((t_{1}t_{2})^{-2}\ln(t_{1}t_{2}))\big],
    \label{A2-2}\\
\Big\langle\frac{\partial \delta_{P_{i},t_{i}}}{\partial t_{i}},\frac{\partial \delta_{P_{j},t_{j}}}{\partial t_{j}}\Big\rangle
&= 2^{n-4} (n-2)^{3}|\SSphere^{n-1}|(t_{1}t_{2})^{-\frac{n}{2}}\big[1+O_{n}((t_{1}t_{2})^{-2}\ln(t_{1}t_{2}))\big].
    \label{A2-3}
\end{align}
\end{lem}

\begin{proof}
We will only  consider $i = 1$ and $j = 2$. Let $T = t_1t_2$. Proceeding as in the proof of Lemma \ref{A1}, we find
\begin{align*}
\langle\delta_{P_{1},t_{1}},\delta_{P_{2},t_{2}}\rangle
    &= 2^{n-2}n(n-2) \int_{\mathbb{R}^n} \frac{t_1^{-\frac{n+2}{2}}t_2^{\frac{n-2}{2}}}{(1 + t_1^{-2}|y|^2)^{\frac{n+2}{2}}(1 + t_2^2|y|^2)^{\frac{n-2}{2}}}dy,
\\
\Big\langle\frac{\partial \delta_{P_{1},t_{1}}}{\partial t_{1}},\delta_{P_{2},t_{2}}\Big\rangle
&=- 2^{n-3} n (n+2)(n-2)\int_{\mathbb{R}^n}\frac{t_1^{-\frac{n+4}{2}}t_2^{\frac{n-2}{2}}(1 - t_1^{-2}|y|^2)}{(1 + t_1^{-2}|y|^2)^{\frac{n+4}{2}}(1 + t_2^2|y|^2)^{\frac{n-2}{2}}}dy,\\
\Big\langle\frac{\partial \delta_{P_{1},t_{1}}}{\partial t_{1}},\frac{\partial \delta_{P_{2},t_{2}}}{\partial t_{2}}\Big\rangle
&= - 2^{n-4} n(n+2)(n-2)^2 \int_{\mathbb{R}^n}\frac{t_1^{-\frac{n+4}{2}}t_2^{\frac{n-4}{2}}(1 - t_1^{-2}|y|^2)(1 - t_2^2|y|^2)}{(1 + t_1^{-2}|y|^2)^{\frac{n+4}{2}}(1 + t_2^2|y|^2)^{\frac{n}{2}}}dy.
\end{align*}
Making the change of variable $z = t_2 y$ and switching to spherical coordinates, we get
\begin{align}
\langle\delta_{P_{1},t_{1}},\delta_{P_{2},t_{2}}\rangle
    &= 2^{n-2}n(n-2)|\mathbb{S}^{n-1}| T^{-\frac{n+2}{2}} \int_0^\infty \frac{r^{n-1}dr}{(1 + T^{-2}r^2)^{\frac{n+2}{2}}(1 + r^2)^{\frac{n-2}{2}}},\label{A2-P1}
\\
\Big\langle\frac{\partial \delta_{P_{1},t_{1}}}{\partial t_{1}},\delta_{P_{2},t_{2}}\Big\rangle
&=- 2^{n-3} n(n+2)(n-2) |\mathbb{S}^{n-1}|t_1^{-1}T^{-\frac{n+2}{2}}\int_0^\infty\frac{r^{n-1}(1 - T^{-2} r^2)dr}{(1 + T^{-2}r^2)^{\frac{n+4}{2}}(1 + r^2)^{\frac{n-2}{2}}},\label{A2-P2}\\
\Big\langle\frac{\partial \delta_{P_{1},t_{1}}}{\partial t_{1}},\frac{\partial \delta_{P_{2},t_{2}}}{\partial t_{2}}\Big\rangle
&= - 2^{n-4} n(n+2)(n-2)^2 |\mathbb{S}^{n-1}|T^{-\frac{n+4}{2}} \int_{0}^\infty\frac{r^{n-1}(1 - T^{-2} r^2) (1 - r^2)dr}{(1 + T^{-2}r^2)^{\frac{n+4}{2}}(1 + r^2)^{\frac{n}{2}}}.\label{A2-P3}
\end{align}
Considering the cases $r \leq 1$, $1 \leq r \leq T$ and $r \geq T$, we observe that the integrands in \eqref{A2-P1}--\eqref{A2-P3} can be rewritten respectively as
\[
\frac{r}{(1 + T^{-2}r^2)^{\frac{n+2}{2}}} + \textrm{error},  \frac{r(1 - T^{-2} r^2)}{(1 + T^{-2}r^2)^{\frac{n+4}{2}}} + \textrm{error},  \text{ and }  - \frac{r (1 - T^{-2} r^2) }{(1 + T^{-2}r^2)^{\frac{n+4}{2}}} + \textrm{error}.
\]
where the error terms of size $O(\min\{1, r^{-1}, T^{n+2}r^{-n-3}\})$. Therefore, the integrals in \eqref{A2-P1}--\eqref{A2-P3} are respectively
\begin{align*}
\int_{0}^\infty \frac{r dr}{(1 + T^{-2}r^2)^{\frac{n+2}{2}}} + O(\ln T)
    &= \frac{1}{n}T^2 + O(\ln T), \\
\int_{0}^\infty \frac{r (1 - T^{-2} r^2) dr}{(1 + T^{-2}r^2)^{\frac{n+4}{2}}} + O(\ln T)
    &= \frac{n-2}{n(n+2)} T^2 + O(\ln T),\\
- \int_0^\infty \frac{r(1 - T^{-2} r^2) dr}{(1 + T^{-2}r^2)^{\frac{n+4}{2}}}+ O(\ln T)
    &= -\frac{n-2}{n(n+2)} T^2 + O(\ln T),
\end{align*}
where the $O(\ln T)$ comes from integrating the error terms in $[1,T]$. Returning to \eqref{A2-P1}--\eqref{A2-P3}, we obtain the conclusion.
\end{proof}

\begin{lem}\label{L2}
Let $\alpha$, $\beta>0$ satisfy $\alpha+\beta=\frac{2n}{n-2}$. If $t_{1}t_{2}\geq e$, we have for $i \neq j$ that
\begin{equation*}
\int_{\SSphere^{n}}\delta_{P_{i},t_{i}}^{\alpha}\delta_{P_{j},t_{j}}^{\beta}\leq C(n,\alpha,\beta)\begin{cases}
    (t_{1}t_{2})^{-\frac{\theta(n-2)}{2}} & \text{ if } \theta < \frac{n}{n-2},\\
    (t_1t_2)^{-\frac{n}{2}}\ln(t_{1}t_{2}) &\text{ if } \theta = \frac{n}{n-2},
    \end{cases}
\end{equation*}
where $\theta:=\min\{\alpha,\beta\}$.

\end{lem}

\begin{proof} We will only consider $i = 1$ and $j = 2$ and $\alpha \leq \beta$, so that $\theta = \alpha$. Proceeding as in the start of the proof of Lemma \ref{A2}, we find
\begin{align*}
\int_{\SSphere^{n}}\delta_{P_{1},t_{1}}^{\alpha}\delta_{P_{2},t_{2}}^{\beta}
	&= 2^n |\mathbb{S}^{n-1}| T^{-\frac{(n-2)\beta}{2}}\int_0^\infty \frac{  r^{n-1}}{(1 +  r^2)^{\frac{(n-2)\alpha}{2}}(1 + T^{-2}  r^2)^{\frac{(n-2)\beta}{2}}}\, dr.
\end{align*}
The integrand is 
\[
\begin{cases}
    O(1) & \text{ for } r \leq 1,\\
    O(r^{n-1-(n-2)\alpha}) & \text{ for }1 < r < T,\\
    O(T^{(n-2)\beta}r^{-n-1}) & \text{ for } r \geq T.
\end{cases}
\]
The conclusion follows.
\end{proof}

\begin{lem}\label{R1}
For $t_i \geq 1$,
\begin{equation}
\int_{\mathbb{S}^n} K \delta_{P_i,t_i}^{\frac{2n}{n-2}}
    = \frac{4c_0 K_i}{n(n-2)} + \frac{2^{2n-3}}{n-1} |\mathbb{S}^{n-1}| a_i t_i^{-(n-2)}    
        + O(1) t_i^{-\mu}.
        \label{R1-1}
\end{equation}
\end{lem}

\begin{proof}
We will only consider $i = 1$. We have
\[
\int_{\mathbb{S}^n} K \delta_{P_1,t_1}^{\frac{2n}{n-2}}
    = 2^n t_1^{-n}\int_{\mathbb{R}^n} \frac{K(|y|)dy}{(1 + t_1^{-2}|y|^2)^n}
\]
Making the change of variable $z = t_1^{-1}y$ and switching to spherical coordinates, we get
\[
\int_{\mathbb{S}^n} K \delta_{P_1,t_1}^{\frac{2n}{n-2}}
    = 2^n |\mathbb{S}^{n-1}|\int_0^\infty \frac{K(t_1 r) r^{n-1}dr}{(1 + r^2)^n}
\]
Now, by \eqref{higher0X},
\[
K(s) = K_1 + 2^{n-2}a_1 s^{-(n-2)} + \begin{cases}
    O(1) &\text{ for } s \leq 1,\\
    O(1)s^{-\mu_1} & \text{ for } s > 1.
\end{cases}
\]
Thus
\begin{align*}
\int_{\mathbb{S}^n} K \delta_{P_1,t_1}^{\frac{2n}{n-2}}
    &= 2^n |\mathbb{S}^{n-1}|\int_0^\infty \frac{K_1 r^{n-1}dr}{(1 + r^2)^n}
        + 2^{2n-2}a_1 t_1^{-(n-2)} |\mathbb{S}^{n-1}|\int_0^\infty \frac{rdr}{(1 + r^2)^n}\\
        &\qquad
        + O(1) \int_0^{t_1^{-1}} r^{n-1}dr
        + O(1) t_1^{-\mu_1} \int_{t_1^{-1}} \frac{r^{n-1-\mu_1}dr}{(1+r^2)^n}\\
    &= 2^n |\mathbb{S}^{n-1}|\int_0^\infty \frac{K_1 r^{n-1}dr}{(1 + r^2)^n}
        + \frac{2^{2n-3}}{n-1}a_1 t_1^{-(n-2)} |\mathbb{S}^{n-1}|   
        + O(1) t_1^{-\mu_1}.
\end{align*}
The conclusion follows from \eqref{AB1-}.
\end{proof}

\begin{lem}\label{R2}
For $t_1, t_2 \geq e$ and $i \neq j$,
\begin{align}
\int_{\mathbb{S}^n} K \delta_{P_i,t_i}^{\frac{n}{n-2}}\delta_{P_j,t_j} 
    &=  \frac{2^n}{n} |\mathbb{S}^{n-1}| K_i(t_1t_2)^{-\frac{n-2}{2}} \times \nonumber\\
        &\qquad \times \left(1+O ((t_{1}t_{2})^{-2}\ln(t_{1}t_{2}))+\begin{cases}
O(t_{i}^{-2}\ln t_{i})& \text{ if }n=4,\\
O(t_{i}^{-2})& \text{ if }n\geq 5
\end{cases}\right)
        \label{R2-1}
\end{align}
\end{lem}

\begin{proof}
We will only consider $i = 1$ and $j = 2$. Let $T = t_1t_2$. Since $-L_{g_0} \delta_{P,t} = \frac{n(n-2)}{4} \delta_{P,t}^{\frac{n+2}{n-2}}$, we have by \eqref{A2-1} in Lemma \ref{A2} that
\[
\int_{\mathbb{S}^n}   \delta_{P_1,t_1}^{\frac{n+2}{n-2}} \delta_{P_2,t_2}
    = \frac{4}{n(n-2)} \langle \delta_{P_1,t_1}, \delta_{P_2,t_2}\rangle = \frac{2^n}{n} |\mathbb{S}^{n-1}| T^{-\frac{n-2}{2}}[1 + O(T^{-2}\ln T)].
\]
On the other hand, computing as in the proof of Lemma \ref{A2}, we have
\[
\int_{\mathbb{S}^n} (K - K_1) \delta_{P_1,t_1}^{\frac{n+2}{n-2}} \delta_{P_2,t_2}
    = 2^n |\mathbb{S}^{n-1}| T^{-\frac{n+2}{2}}\int_0^\infty \frac{(K(r/t_2) - K_1) r^{n-1}dr}{(1 + T^{-2}r^2)^{\frac{n+2}{2}}(1 + r^2)^{\frac{n-2}{2}}}.
\]
Therefore to obtain \eqref{R2-1}, it suffices to show that
\begin{equation}
\int_0^\infty \frac{|K(r/t_2) - K_1|r^{n-1}dr}{(1 + T^{-2}r^2)^{\frac{n+2}{2}}(1 + r^2)^{\frac{n-2}{2}}}
    \leq \begin{cases}
O(t_{2}^2\ln t_{1})& \text{ if }n=4,\\
O(t_{2}^2)& \text{ if }n\geq 5.
\end{cases}
    \label{R2-P}
\end{equation}
Indeed, to estimate the integrand, we use \eqref{higher0X} which gives
\[
|K(r/t_2) - K_1| \leq \begin{cases}
    O(1) & \text{ if } 0 \leq r \leq t_2,\\
    O(1) t_2^{n-2} r^{-(n-2)} & \text{ if } r > t_2.
\end{cases}
\]
Hence, the integrand is 
\[
\begin{cases}
O(1) r & \text{ if } 0 \leq r \leq t_2,\\
O(1) t_2^{n-2} r^{-(n-3)} &\text{ if } t_2 < r < T,\\
O(1) T^{n+2} t_2^{n-2} r^{-2n + 1} &\text{ if } r \geq T.
\end{cases}
\]
Integrating we obtain \eqref{R2-P}, which concludes the proof.
\end{proof}

\begin{lem}\label{A3}
\begin{align}
 \int_{\SSphere^{n}}K\delta_{P_{i},t_{i}}^{\frac{n+2}{n-2}}\frac{\partial \delta_{P_{i},t_{i}}}{\partial t_{i}} 
&  =-\frac{2^{2n-4}(n-2)^2}{n(n-1)} |\SSphere^{n-1}|  a_{i} t_{i}^{-(n-1)}\nonumber\\
&\quad -\frac{2^{\mu_{i}+n-1}(n-2)}{n}\mu_{i}|\SSphere^{n-1}|\left(\int_{0}^{\infty}\frac{r^{n+\mu_{i}-1}dr}{(1+r^{2})^{n}}\right)b_{i}t_{i}^{-(\mu_{i}+1)}
    +o(t_{i}^{-(\mu_{i}+1)}),\label{A3-1}
\end{align}
where the little o notation is meant for $t_i \rightarrow \infty$ and with a rate of convergence depending on the function $\omega$ in \eqref{higher0X}.
\end{lem}

\begin{proof}
We will only consider $i = 1$. Since $-L_{g_0} \delta_{P,t} = \frac{n(n-2)}{4}\delta_{P,t}^{\frac{n+2}{n-2}}$,
\[
\int_{\SSphere^{n}}\delta_{P_{1},t_{1}}^{\frac{n+2}{n-2}}\frac{\partial \delta_{P_{1},t_{1}}}{\partial t_{1}}
    = \frac{4}{n(n-2)}\left\langle\delta_{P_1,t_1},\frac{\partial \delta_{P_1,t_1}}{\partial t_1}\right\rangle = 0.
\]
(The last identity can be seen e.g. by differentiating \eqref{AB1-}.) Therefore
\begin{align*}
\int_{\SSphere^{n}} K\delta_{P_{1},t_{1}}^{\frac{n+2}{n-2}}\frac{\partial \delta_{P_{1},t_{1}}}{\partial t_{1}}
    &= \int_{\SSphere^{n}} (K - K_1)\delta_{P_{1},t_{1}}^{\frac{n+2}{n-2}}\frac{\partial \delta_{P_{1},t_{1}}}{\partial t_{1}}\\
    &= -2^{n-1} (n-2) t_1^{-n-1}\int_{\mathbb{R}^n} \frac{(K(|y|) - K_1)(1 - t_1^{-2}|y|^2)}{(1 + t_1^{-2}|y|^2)^{n+1}} dy.
\end{align*}
Making the change of variables $z = t_1^{-1}y$ and switching to spherical coordinates, we get
\begin{equation}
\int_{\SSphere^{n}} K\delta_{P_{1},t_{1}}^{\frac{n+2}{n-2}}\frac{\partial \delta_{P_{1},t_{1}}}{\partial t_{1}}
    = -2^{n-1} (n-2) |\mathbb{S}^{n-1}| t_1^{-1} \int_0^\infty \frac{(K(t_1r) - K_1)(1 - r^2)r^{n-1}}{(1 + r^2)^{n+1}}dr.
    \label{A3-P1}
\end{equation}

Fix an arbitrary $\varepsilon > 0$. In the sequel the implicit constants in the big $O$ terms are independent of $\varepsilon$. By \eqref{higher0X}, there exists $C_0 = C_0(\varepsilon,\omega)$ such that
\begin{equation}
|\bar R_1(s)| = |K(s) - K_1 - 2^{n-2}a_1 s^{-(n-2)} - 2^{\mu_1} b_1 s^{-\mu_1}| \leq \varepsilon s^{-\mu_1} \text{ for } s \geq C_0,
    \label{A3-P2}
\end{equation}
and there exists $C_1 = C_1(\omega)$
\begin{equation}
|K(s) - K_1| \leq C_1 \text{ for } s < C_0.
     \label{A3-P3}
\end{equation}
We assume in the sequel that $t_1 \geq 2C_0$

We then split the integral on the right hand side of \eqref{A3-P1} as
\begin{align*}
I   
    &:= \int_0^\infty \frac{(K(t_1r) - K_1)(1 - r^2)r^{n-1}}{(1 + r^2)^{n+1}}dr\\
    &= \int_{C_0t_1^{-1}}^\infty \frac{2^{n-2}a_1t_1^{-(n-2)}(1 - r^2)r}{(1 + r^2)^{n+1}}dr
        + \int_{C_0t_1^{-1}}^\infty \frac{2^{\mu_1}b_1t_1^{-\mu_1}(1 - r^2)r^{n-1-\mu_1}}{(1 + r^2)^{n+1}}dr\\
        &\qquad
        + \int_{C_0t_1^{-1}}^\infty \frac{\bar R_1(t_1r) (1 - r^2)r^{n-1}}{(1 + r^2)^{n+1}}dr
        + \int_0^{C_0t_1^{-1}} \frac{(K(t_1r) - K_1)(1 - r^2)r^{n-1}}{(1 + r^2)^{n+1}}dr\\
    &=: I_1 + I_2 + I_3 + I_4
\end{align*}

We compute $I_1$ using integration by parts,
\begin{align*}
I_1
    &= 2^{n-2}a_1t_1^{-(n-2)}\int_{0}^\infty \frac{(1 - r^2)r}{(1 + r^2)^{n+1}}dr + O(1) C_0^2 t_1^{-n}\\
    &= -\frac{2^{n-3}}{n}a_1t_1^{-(n-2)} \int_{0}^\infty  (1 - r^2)   \frac{d}{dr}\Big(\frac{1}{(1 + r^2)^n} \Big) dr + O(1) C_0^2 t_1^{-n}\\
    &= -\frac{2^{n-3}}{n}a_1t_1^{-(n-2)}\Big[- 1 + \int_{0}^\infty   \frac{2r}{(1 + r^2)^n}  dr\Big] + O(1) C_0^2 t_1^{-n}\\
    &= \frac{2^{n-3}(n-2)}{n(n-1)}a_1t_1^{-(n-2)} + O(1) C_0^2 t_1^{-n}.
\end{align*}

For $I_2$, we start with
\[
I_2
    = 2^{\mu_1}b_1t_1^{-\mu_1}\int_{0}^\infty \frac{(1 - r^2)r^{n-1-\mu_1}}{(1 + r^2)^{n+1}}dr + O(1) C_0^{n-\mu_1}t_1^{-n},
\]
and then use the chain of identities
\[
\int_{0}^\infty \frac{r^{n+1-\mu_1}}{(1 + r^2)^{n+1}}dr
    = -\frac{1}{2n}\int_{0}^\infty r^{n-\mu_1}  \frac{d}{dr}\Big(\frac{1}{(1 + r^2)^n} \Big) dr
    = \frac{n-\mu_1}{2n}\int_{0}^\infty\frac{r^{n-1-\mu_1}}{(1 + r^2)^n}\,dr
\]
to obtain
\[
I_2 = \frac{2^{\mu_1-1}}{n} \mu_1 b_1 t_1^{-\mu_1}\int_{0}^\infty \frac{r^{n-1-\mu_1}}{(1 + r^2)^n}dr + O(1) C_0^{n-\mu_1}t_1^{-n}.
\]
Perform an inversion $r \mapsto r^{-1}$, we find
\[
I_2 = \frac{2^{\mu_1-1}}{n} \mu_1 b_1 t_1^{-\mu_1}\int_{0}^\infty \frac{r^{n-1+\mu_1}}{(1 + r^2)^n}dr + O(1) C_0^{n-\mu_1}t_1^{-n}.
\]

By \eqref{A3-P2},
\[
|I_3| \leq \int_{C_0 t_1^{-1}}^\infty\frac{\varepsilon t_1^{-\mu_1}|1 - r^2|r^{n-1-\mu_1}}{(1 + r^2)^{n+1}}dr = O(1) \varepsilon t_1^{-\mu}.
\]

By \eqref{A3-P3},
\[
|I_4| \leq \int_0^{C_0 t_1^{-1}} O(1)r^{n-1}dr = O(1)C_0^n t_1^{-n}.
\]

Putting the above estimates for $I_1, \ldots, I_4$ into \eqref{A3-P1}, we arrive at
\begin{align*}
\int_{\SSphere^{n}}K\delta_{P_{1},t_{1}}^{\frac{n+2}{n-2}}\frac{\partial \delta_{P_{1},t_{1}}}{\partial t_{1}} 
&  =-\frac{2^{2n-4}(n-2)^2}{n(n-1)} |\SSphere^{n-1}|  a_{1} t_{1}^{-(n-1)}\nonumber\\
&\quad\quad-\frac{2^{\mu_{1}+n-1}(n-2)}{n}\mu_{1}|\SSphere^{n-1}|\left(\int_{0}^{\infty}\frac{r^{n+\mu_{i}-1}dr}{(1+r^{2})^{n}}\right)b_{1}t_{1}^{-(\mu_{1}+1)}\\
    &\quad\quad
    + O(1) \big[\varepsilon + (C_0^2  + C_0^{n-\mu_1} + C_0^n) t_1^{-(n-\mu_1)}\big]t_{1}^{-(\mu_{1}+1)}, 
\end{align*}
In particular, for all sufficiently large $t_1$, the last error term is $O(1) \varepsilon t_{1}^{-(\mu_{1}+1)}$. As $\varepsilon$ is arbitrary, the lemma is proved.
\end{proof}

\begin{lem}\label{A4}
There exists a constant $C=C(n,K)>0$ depending on $n$ and $K$ such that if $t_{1}$, $t_{2}>C$, we have for $i \neq j$ that
\begin{align}
&\int_{\SSphere^{n}}K\delta_{P_{i},t_{i}}^{\frac{n+2}{n-2}}\frac{\partial \delta_{P_{j},t_{j}}}{\partial t_{j}}\nonumber\\
&\quad=-\frac{2^{n-1}(n-2)}{n}|\SSphere^{n-1}|K_{i}t_{j}^{-1}(t_{1}t_{2})^{-\frac{n-2}{2}} \times\nonumber\\
    &\qquad \times\left(1+O_{n}((t_{1}t_{2})^{-2}\ln(t_{1}t_{2}))+\begin{cases}
O(t_{i}^{-2}\ln t_{i})& \text{ if }n=4,\\
O(t_{i}^{-2})& \text{ if }n\geq 5
\end{cases}\right),\label{A4-1}
\end{align}
and
\begin{align}
&\int_{\SSphere^{n}}K\delta_{P_{i},t_{i}}^{\frac{4}{n-2}}\delta_{P_{j},t_{j}}\frac{\partial \delta_{P_{i},t_{i}}}{\partial t_{i}}\nonumber\\
&\quad=-\frac{2^{n-1}(n-2)^{2}}{n(n+2)}|\SSphere^{n-1}|K_{i}t_{i}^{-1}(t_{1}t_{2})^{-\frac{n-2}{2}} \times\nonumber\\
&\qquad \times \left(1+O_{n}((t_{1}t_{2})^{-2}\ln(t_{1}t_{2}))+\begin{cases}
O(t_{i}^{-2}\ln t_{i})&\text{ if }n=4,\\
O(t_{i}^{-2})& \text{ if }n\geq 5.
\end{cases}\right).\label{A4-2}
\end{align}
\end{lem}

\begin{proof}
We will only consider the case $i = 1$ and $j = 2$. Let $T = t_1t_2$. 

Recall that $-L_{g_0} \delta_{P,t} = \frac{n(n-2)}{4}\delta_{P,t}^{\frac{n+2}{n-2}}$ and $-L_{g_0} \frac{\partial\delta_{P,t}}{\partial t} = \frac{n(n+2)}{4}\delta_{P,t}^{\frac{4}{n-2}}\frac{\partial\delta_{P,t}}{\partial t}$. Thus, by \eqref{A2-2} in Lemma \ref{A2},
\begin{align*}
\int_{\mathbb{S}^n} K_1 \delta_{P_1,t_1}^{\frac{n+2}{n-2}}\frac{\partial \delta_{P_2,t_2}}{\partial t_2}
    &= \frac{4K_1}{n(n-2)}\left\langle\delta_{P_1,t_1},\frac{\partial \delta_{P_2,t_2}}{\partial t_2}\right\rangle\\
    &= - \frac{2^{n-1} (n-2)}{n} |\SSphere^{n-1}| K_1 t_{2}^{-1}T^{-\frac{n-2}{2}}   \big[1+O_{n}(T^{-2}\ln T)\big],\\
\int_{\SSphere^{n}}K_1\delta_{P_{1},t_{1}}^{\frac{4}{n-2}}\delta_{P_{2},t_{2}}\frac{\partial \delta_{P_{1},t_{1}}}{\partial t_{1}} 
     &= \frac{4K_1}{n(n+2)}\left\langle\frac{\partial \delta_{P_1,t_1}}{\partial t_1},\delta_{P_2,t_2}\right\rangle\\
    & =-\frac{2^{n-1}(n-2)^{2}}{n(n+2)}|\SSphere^{n-1}|K_1t_1^{-1}T^{-\frac{n-2}{2}}  \big[1+O_{n}(T^{-2}\ln T)\big].
\end{align*}
On the other hand, computing as in the proof of Lemma \ref{A2}, we find
\begin{align*}
\Big|\int_{\SSphere^{n}}(K-K_1)\delta_{P_{1},t_{1}}^{\frac{n+2}{n-2}}\frac{\partial \delta_{P_{2},t_{2}}}{\partial t_{2}}\Big|
    &\leq O(1) t_2^{-1} T^{-\frac{n+2}{2}}\int_0^\infty \frac{|K(r/t_2) - K_1|r^{n-1}dr}{(1 + T^{-2}r^2)^{\frac{n+2}{2}}(1 + r^2)^{\frac{n-2}{2}}},\\
\Big|\int_{\SSphere^{n}}(K- K_1)\delta_{P_{1},t_{1}}^{\frac{4}{n-2}}\delta_{P_{2},t_{2}}\frac{\partial \delta_{P_{1},t_{1}}}{\partial t_{1}}
    &\leq O(1) t_1^{-1} T^{-\frac{n+2}{2}}\int_0^\infty \frac{|K(r/t_2) - K_1|r^{n-1}dr}{(1 + T^{-2}r^2)^{\frac{n+2}{2}}(1 + r^2)^{\frac{n-2}{2}}}.
\end{align*}
Recalling the bound \eqref{R2-P} for the integral on the right hand side, we conclude the proof.
\end{proof}

\begin{lem}\label{A16}
If $t_{1}t_{2}\geq e$, we have for $i \neq j$ that
\begin{align}
&\int_{\SSphere^{n}}K\Big|(\sum\limits_{l=1}^{2}\bar{\alpha}_{l}\delta_{P_{l},t_{l}})^{\frac{n+2}{n-2}}-\sum\limits_{l=1}^{2}(\bar{\alpha}_{l}\delta_{P_{l},t_{l}})^{\frac{n+2}{n-2}}-\frac{n+2}{n-2}(\bar{\alpha}_{i}\delta_{P_{i},t_{i}})^{\frac{4}{n-2}}(\bar{\alpha}_{j}\delta_{P_{j},t_{j}})\Big| \Big|\frac{\partial \delta_{P_{i},t_{i}}}{\partial t_{i}}\Big|\nonumber\\
&\quad \leq  
    C(n,K)t_{i}^{-1}(t_{1}t_{2})^{-\frac{n}{2}}\ln(t_{1}t_{2}).\label{A16-1}
\end{align}
\end{lem}

\begin{proof}
First, recall that 
\begin{equation}
\Big|\frac{\partial \delta_{P_{i},t_{i}}}{\partial t_{i}}\Big| \leq C(n)t_i^{-1} \delta_{P_i,t_i}.
    \label{A16-P1}
\end{equation}
Next, note the inequality 
\[
|(a + b)^p - a^p - b^p  - p a^{p-1} b| 
    \leq \begin{cases}
    C(p)a^{\min\{1,p-1\}}b^{\max\{1,p-1\}} & \text{ for } a \leq b,\\
    C(p)a^{\max\{0,p-2\}}b^{\min\{2,p\}} & \text{ for } a > b
\end{cases} 
\]
for $a, b \geq 0$, $p \geq 1$. In particular, for $p = \frac{n+2}{n-2} \in (1,3]$, we have
\begin{equation}
\Big|(a + b)^{\frac{n+2}{n-2}} - a^{\frac{n+2}{n-2}} - b^{\frac{n+2}{n-2}}  - \frac{n+2}{n-2} a^{\frac{4}{n-2}} b\Big| 
    \leq C(n)a^{\frac{2}{n-2}}b^{\frac{n}{n-2}}.
    \label{IneqX}
\end{equation}
Consequently,
\begin{equation}
\Big|(\sum\limits_{l=1}^{2}\bar{\alpha}_{l}\delta_{P_{l},t_{l}})^{\frac{n+2}{n-2}}-\sum\limits_{l=1}^{2}(\bar{\alpha}_{l}\delta_{P_{l},t_{l}})^{\frac{n+2}{n-2}}-\frac{n+2}{n-2}(\bar{\alpha}_{i}\delta_{P_{i},t_{i}})^{\frac{4}{n-2}}(\bar{\alpha}_{j}\delta_{P_{j},t_{j}})\Big|  \leq    O(1)\delta_{P_{i},t_{i}}^{\frac{2}{n-2}} \delta_{P_{j},t_{j}}^{\frac{n}{n-2}} . 
    \label{A16-P2}
\end{equation}
The conclusion follows from \eqref{A16-P1}, \eqref{A16-P2} and Lemma \ref{L2}. 
\end{proof}

\begin{lem}\label{A17}
There exists a positive constant $C=C(n,K)$ depending on $n$ and $K$ such that if $t_{1}$, $t_{2}>C$, we have that
\begin{align}
&\int_{\SSphere^{n}}K\Big(\sum\limits_{l=1}^{2}\bar{\alpha}_{l}\delta_{P_{l},t_{l}}\Big)^{\frac{n+2}{n-2}}\delta_{P_{i},t_{i}}\nonumber\\
&\quad=\frac{c_{0}}{c(n)}\bar{\alpha}_{i}+\frac{2^{n+1}n(n-1)}{n-2}|\SSphere^{n-1}|\bar{\alpha}_{j}(t_{1}t_{2})^{-\frac{n-2}{2}}  + \frac{2^{2n-3}}{n-1}|\SSphere^{n-1}|\,a_{i}\bar{\alpha}_{i}^{\frac{n+2}{n-2}}t_{i}^{-(n-2)}\nonumber\\
&\quad\quad+O(1)t_{i}^{-\mu_{i}}
    +O(1)(t_{1}t_{2})^{-\frac{n}{2}}\ln(t_{1}t_{2})
    \nonumber\\
&\quad\quad + O(1)(t_{1}t_{2})^{-\frac{n-2}{2}}\begin{cases}
\sum\limits_{l=1}^{2}t_{l}^{-2}\ln t_{l}& \text{ if }n=4,\\
\sum\limits_{l=1}^{2}t_{l}^{-2}&\text{ if }n\geq5,
\end{cases}\label{hongbing}
\\
&
\int_{\SSphere^{n}}K\Big(\sum\limits_{l=1}^{2}\bar{\alpha}_{l}\delta_{P_{l},t_{l}}\Big)^{\frac{4}{n-2}}\delta_{P_{i},t_{i}}^{2} =\frac{c_{0}}{c(n)}+O(1)t_{i}^{-(n-2)}
    + O(1)(t_{1}t_{2})^{-\frac{n-2}{2}} .
    \label{hongbing-}
\end{align}
and, for $i \neq j$,
\begin{align}
&\int_{\SSphere^{n}}K\Big(\sum\limits_{l=1}^{2}\bar{\alpha}_{l}\delta_{P_{l},t_{l}}\Big)^{\frac{4}{n-2}}\delta_{P_{i},t_{i}}\delta_{P_{j},t_{j}}\nonumber\\
&\quad= 2^{n+1}(n-1)|\SSphere^{n-1}|(t_{1}t_{2})^{-\frac{n-2}{2}} \times\nonumber\\
&\quad\quad \times \left(1+\begin{cases}
O((t_{1}t_{2})^{-1}\ln(t_{1}t_{2})+\sum\limits_{l=1}^{2}t_{l}^{-2}\ln t_{l}) &\text{ if }n=4,\\
O((t_{1}t_{2})^{-1}\ln(t_{1}t_{2})+\sum\limits_{l=1}^{2}t_{l}^{-2}) &\text{ if }n\geq5
\end{cases}\right).\label{hongbing--}
\end{align}
\end{lem}

\begin{proof}We will only consider $i = 1$ and $j = 2$. Let $T = t_1t_2$.

\medskip
\noindent
\underline{Proof of \eqref{hongbing}.} We split
\begin{align*}
&\int_{\SSphere^{n}}K\Big(\sum\limits_{l=1}^{2}\bar{\alpha}_{l}\delta_{P_{l},t_{l}}\Big)^{\frac{n+2}{n-2}}\delta_{P_{1},t_{1}}\\
   &\quad= \int_{\SSphere^{n}} K  \bar{\alpha}_{1}^{\frac{n+2}{n-2}}  \delta_{P_{1},t_{1}}^{\frac{2n}{n+2}}
            + \int_{\SSphere^{n}} K  \bar{\alpha}_{2}^{\frac{n+2}{n-2}}\delta_{P_{2},t_{2}}^{\frac{n+2}{n-2}}  \delta_{P_{1},t_{1}} 
            + \frac{n+2}{n-2} \int_{\SSphere^{n}} K \bar\alpha_1^{\frac{4}{n-2}} \bar \alpha_2 \delta_{P_{1},t_{1}}^{\frac{n+2}{n-2}}\delta_{P_{2},t_{2}} \\
       &\qquad+ \int_{\SSphere^{n}} K\Big[\Big(\sum\limits_{l=1}^{2}\bar{\alpha}_{l}\delta_{P_{l},t_{l}}\Big)^{\frac{n+2}{n-2}} 
            - \sum\limits_{l=1}^{2}\bar{\alpha}_{l}^{\frac{n+2}{n-2}}\delta_{P_{l},t_{l}}^{\frac{n+2}{n-2}} 
            - \frac{n+2}{n-2} \bar\alpha_1^{\frac{4}{n-2}}\bar \alpha_2 \delta_{P_{1},t_{1}}^{\frac{4}{n-2}}\delta_{P_{2},t_{2}} \Big] \delta_{P_{1},t_{1}}\\
    &\quad=: I_1 + I_2 + I_3 + I_4.
\end{align*}

Using \eqref{R1-1} in Lemma \ref{R1} and recalling that $K_l = n(n-1)\bar\alpha_l^{-\frac{4}{n-2}}$, we have
\begin{equation}
I_1 =  \frac{ c_0  }{ c(n)}  \bar{\alpha}_{1} + \frac{2^{2n-3}}{n-1}  |\mathbb{S}^{n-1}| \bar{\alpha}_{1}^{\frac{n+2}{n-2}} a_1 t_1^{-(n-2)}  
        + O(1) t_1^{-\mu}.
        \label{A17-I1}
\end{equation}
By \eqref{R2-1} in Lemma \ref{R2}, we have
\begin{align}
I_2
    &= 2^n (n-1) |\mathbb{S}^{n-1}| \bar{\alpha}_{2}  T^{-\frac{n-2}{2}}  \nonumber\\
        &\qquad   + O(1) T^{-\frac{n+2}{2}}\ln T + O(1) T^{-\frac{n-2}{2}} \begin{cases}
O(t_{2}^{-2}\ln t_{2})& \text{ if }n=4,\\
O(t_{2}^{-2})& \text{ if }n\geq 5,
\end{cases} \label{A17-I2}\\
I_3
    &= \frac{2^n(n-1)(n+2)}{n-2}    |\mathbb{S}^{n-1}| \bar \alpha_2  T^{-\frac{n-2}{2}}  \nonumber\\
        &\qquad + O(1) T^{-\frac{n+2}{2}}\ln T + O(1) T^{-\frac{n-2}{2}} \begin{cases}
O(t_{1}^{-2}\ln t_{1})& \text{ if }n=4,\\
O(t_{1}^{-2})& \text{ if }n\geq 5.
\end{cases} \label{A17-I3}
\end{align}
Next, using \eqref{IneqX}, we have
\[
\Big| \Big(\sum\limits_{l=1}^{2}\bar{\alpha}_{l}\delta_{P_{l},t_{l}}\Big)^{\frac{n+2}{n-2}} 
            - \sum\limits_{l=1}^{2}\bar{\alpha}_{l}^{\frac{n+2}{n-2}}\delta_{P_{l},t_{l}}^{\frac{n+2}{n-2}} 
            - \frac{n+2}{n-2} \bar\alpha_1^{\frac{4}{n-2}}\bar \alpha_2 \delta_{P_{1},t_{1}}^{\frac{4}{n-2}}\delta_{P_{2},t_{2}} \Big| \leq O(1) \delta_{P_1,t_1}^{\frac{2}{n-2}}\delta_{P_2,t_2}^{\frac{n}{n-2}}.
\]
Thus, by Lemma \ref{L2},
\[
|I_4|
    \leq O(1) T^{-\frac{n}{2}}\ln T.
\]
Estimate \eqref{hongbing} is readily seen.

\medskip
\noindent
\underline{Proof of \eqref{hongbing-}.} We split
\begin{align}
\int_{\SSphere^{n}}K\Big(\sum\limits_{l=1}^{2}\bar{\alpha}_{l}\delta_{P_{l},t_{l}}\Big)^{\frac{4}{n-2}}\delta_{P_{1},t_{1}}^{2}
    &= \int_{\SSphere^{n}}K \bar{\alpha}_{1}^{\frac{4}{n-2}} \delta_{P_{1},t_{1}}^{\frac{2n}{n-2}}\nonumber\\
      &\qquad + \int_{\SSphere^{n}} K\Big[\Big(\sum\limits_{l=1}^{2}\bar{\alpha}_{l}\delta_{P_{l},t_{l}}\Big)^{\frac{4}{n-2}} - \bar{\alpha}_{1}^{\frac{4}{n-2}} \delta_{P_{1},t_{1}}^{\frac{4}{n-2}}\Big]\delta_{P_{1},t_{1}}^{2}\nonumber\\
     &=: \bar\alpha_1^{-1} I_1 + I_5 .
     \label{A17-Split2}
\end{align}
Using the inequality
\[
|(a + b)^p - a^p| \leq \begin{cases} 
    C_p (a^{p-1}b + b^p) & \text{ for } p > 1,\\
    a^{p-1}b & \text{ for } 0 < p \leq 1,
    \end{cases}
\]
where $a, b > 0$, we have
\[
\Big|\Big(\sum\limits_{l=1}^{2}\bar{\alpha}_{l}\delta_{P_{l},t_{l}}\Big)^{\frac{4}{n-2}} - \bar{\alpha}_{1}^{\frac{4}{n-2}} \delta_{P_{1},t_{1}}^{\frac{4}{n-2}}\Big|
    \leq \begin{cases}
        O(1) \delta_{P_{1},t_{1}}^{\frac{6-n}{n-2}} \delta_{P_{2},t_{2}}  + O(1) \delta_{P_{2},t_{2}}^{\frac{4}{n-2}} & \text{ for } n = 4, 5,\\
        O(1) \delta_{P_{1},t_{1}}^{\frac{6-n}{n-2}}\delta_{P_{2},t_{2}}  & \text{ for } n \ge 6.
    \end{cases}
\]
Thus, by Lemma \ref{L2},
\[
|I_5| \leq  
O(1)T^{-\frac{n-2}{2}} .
\]
Insert this together with estimate \eqref{A17-I1} for $I_1$ into \eqref{A17-Split2}, we get \ref{hongbing-}.

\medskip
\noindent
\underline{Proof of \eqref{hongbing--}.} We split
\begin{align}
\int_{\SSphere^{n}}K\Big(\sum\limits_{l=1}^{2}\bar{\alpha}_{l}\delta_{P_{l},t_{l}}\Big)^{\frac{4}{n-2}}\delta_{P_{1},t_{1}}\delta_{P_{2},t_{2}}
   & = \int_{\SSphere^{n}}K \bar{\alpha}_{1}^{\frac{4}{n-2}} \delta_{P_{1},t_{1}}^{\frac{n+2}{n-2}}\delta_{P_{2},t_{2}}
        + \int_{\SSphere^{n}}K \bar{\alpha}_{2} \delta_{P_{1},t_{1}}\delta_{P_{2},t_{2}}^{\frac{n+2}{n-2}}\nonumber\\
        &\quad + \int_{\SSphere^{n}}K\Big[\Big(\sum\limits_{l=1}^{2}\bar{\alpha}_{l}\delta_{P_{l},t_{l}}\Big)^{\frac{4}{n-2}} - \sum\limits_{l=1}^{2}\bar{\alpha}_{l}^{\frac{4}{n-2}}\delta_{P_{l},t_{l}}^{\frac{4}{n-2}}\Big]\delta_{P_{1},t_{1}}\delta_{P_{2},t_{2}}\nonumber\\
    &=: \bar\alpha_2^{-1} (I_2 + I_3) + I_6.
    \label{A17-Split3}
\end{align}
To estimate $I_6$, we note the inequality 
\[
|(a+b)^p - a^p - b^p| \leq C(p) a^{\frac{p}{2}} b^{\frac{p}{2}}
\]
for $a,b \geq 0$, $0 < p \leq 2$. In particular, for $p = \frac{4}{n-2}$, we have
\begin{equation}
   |(a+b)^{\frac{4}{n-2}} - a^{\frac{4}{n-2}} - b^{\frac{4}{n-2}}| \leq C(p) a^{\frac{2}{n-2}} b^{\frac{2}{n-2}}, 
    \label{IneqY}
\end{equation}
which implies
\begin{equation}
\Big|\Big(\sum\limits_{l=1}^{2}\bar{\alpha}_{l}\delta_{P_{l},t_{l}}\Big)^{\frac{4}{n-2}} - \sum\limits_{l=1}^{2}\bar{\alpha}_{l}^{\frac{4}{n-2}}\delta_{P_{l},t_{l}}^{\frac{4}{n-2}}\Big| \leq O(1) \delta_{P_{1},t_{1}}^{\frac{2}{n-2}}\delta_{P_{2},t_{2}}^{\frac{2}{n-2}}.
    \label{A17-Aux1}
\end{equation}
Thus, by Lemma \ref{L2},
\[
|I_6| \leq O(1) T^{-\frac{n}{2}} \ln T.
\]
Insert this together with estimates \eqref{A17-I2} and \eqref{A17-I3} for $I_2$ and $I_3$ into \eqref{A17-Split3}, we obtain \eqref{hongbing--}.
\end{proof}

\begin{lem}\label{Axd}
There exists a positive constant $C=C(n,K)$ depending on $n$ and $K$ such that if $t_{1}$, $t_{2}>C$, we have that
\begin{equation}
\int_{\SSphere^{n}}K\Big(\sum\limits_{l=1}^{2}\bar{\alpha}_{l}\delta_{P_{l},t_{l}}\Big)^{\frac{n+2}{n-2}}\frac{\delta_{P_{i},t_{i}}}{\partial t_i}
    = O(1)t_i^{-1}\big(t_i^{-(n-2)} + (t_1t_2)^{-\frac{n-2}{2}}\big).
\label{Axd-1}
\end{equation}
\end{lem}

\begin{proof}
We will only consider $i = 1$. We split the integral on the left hand side of \eqref{Axd-1} as
\begin{align*}
&\int_{\SSphere^{n}}K \bar{\alpha}_{1}^{\frac{n+2}{n-2}}\delta_{P_{1},t_{1}}^{\frac{n+2}{n-2}} \frac{\delta_{P_{1},t_{1}}}{\partial t_1}
        + \int_{\SSphere^{n}}K \bar{\alpha}_{2}^{\frac{n+2}{n-2}}\delta_{P_{2},t_{2}}^{\frac{n+2}{n-2}} \frac{\delta_{P_{1},t_{1}}}{\partial t_1}
        + \frac{n+2}{n-2}\int_{\SSphere^{n}}K \bar{\alpha}_{1}^{\frac{4}{n-2}} \bar \alpha_2 \delta_{P_{1},t_{1}}^{\frac{4}{n-2}} \delta_{P_2,t_2}\frac{\delta_{P_{1},t_{1}}}{\partial t_1}\\
        &\qquad
        + \int_{\SSphere^{n}}K \Big[(\sum\limits_{l=1}^{2}\bar{\alpha}_{l}\delta_{P_{l},t_{l}})^{\frac{n+2}{n-2}}-\sum\limits_{l=1}^{2}(\bar{\alpha}_{l}\delta_{P_{l},t_{l}})^{\frac{n+2}{n-2}}-\frac{n+2}{n-2}(\bar{\alpha}_{i}\delta_{P_{i},t_{i}})^{\frac{4}{n-2}}(\bar{\alpha}_{j}\delta_{P_{j},t_{j}})\Big]\frac{\delta_{P_{1},t_{1}}}{\partial t_1}.
\end{align*}
These four integrals are estimated respectively by \eqref{A3-1} in Lemma \ref{A3}, \eqref{A4-1} and \eqref{A4-2} in Lemma \ref{A4}, and \eqref{A16-1} in Lemma \ref{A16}. The conclusion follows.
\end{proof}

\begin{lem}\label{Ayd}
There exists a positive constant $C=C(n,K)$ depending on $n$ and $K$ such that if $t_{1}$, $t_{2}>C$, we have that
\begin{equation}
\int_{\SSphere^{n}}K\Big(\sum\limits_{l=1}^{2}\bar{\alpha}_{l}\delta_{P_{l},t_{l}}\Big)^{\frac{4}{n-2}}\delta_{P_{i},t_{i}} \frac{\delta_{P_{i},t_{i}}}{\partial t_i}
    = \begin{cases} 
 O(1) t_i^{-1} \big[t_i^{-(n-2)} + (t_1t_2)^{-\frac{n-2}{2}}\big] & \text{ if } n = 4,5,\\
 O(1) t_i^{-1} \big[t_i^{-(n-2)} + (t_1t_2)^{-\frac{n+2}{4}}\big] & \text{ if } n \geq 6.
\end{cases}
\label{Ayd-1}
\end{equation}
\end{lem}

\begin{proof}
We will only consider $i = 1$. We split the integral on the left hand side of \eqref{Axd-1} as
\begin{align*}
\int_{\SSphere^{n}}K \bar{\alpha}_{1}^{\frac{4}{n-2}}\delta_{P_{1},t_{1}}^{\frac{n+2}{n-2}} \frac{\delta_{P_{1},t_{1}}}{\partial t_1}
        + \int_{\SSphere^{n}}K \Big[(\sum\limits_{l=1}^{2}\bar{\alpha}_{l}\delta_{P_{l},t_{l}})^{\frac{4}{n-2}}- \bar{\alpha}_{1}^{\frac{4}{n-2}}\delta_{P_{1},t_{1}}^{\frac{4}{n-2}}\Big]\delta_{P_{1},t_{1}}\frac{\delta_{P_{1},t_{1}}}{\partial t_1} 
        =: I_1 + I_2.
\end{align*}
By \eqref{A3-1} in Lemma \ref{A3}, $|I_1| \leq O(t_1^{-(n-1)})$. By the inequality
\begin{equation*}
    |(a+b)^{p}-a^{p}|a\leq \begin{cases}
        C_{p}\min\{ab^{p},a^pb\}&\text{ if }0 < p<1,\\
        C_{p}(ab^{p} + a^{p}b) &\text{ if }p\geq1
    \end{cases}
\end{equation*}
for $a,b >0$, $p>0$, we have
\begin{equation}
    \Big|\Big(\sum\limits_{l=1}^{2}\bar{\alpha}_{l}\delta_{P_{l},t_{l}}\Big)^{\frac{4}{n-2}}-\bar{\alpha}_{1}^{\frac{4}{n-2}}\delta_{P_{1},t_{1}}^{\frac{4}{n-2}}\Big|\delta_{P_{1},t_{1}}
    \leq \begin{cases}
        O(1) \delta_{P_{1},t_{1}} \delta_{P_{2},t_{2}}^{\frac{4}{n-2}} + \delta_{P_{1},t_{1}}^{\frac{4}{n-2}} \delta_{P_{2},t_{2}}&\text{ if }n=4,5,\\
        O(1) \delta_{P_{1},t_{1}}^{\frac{n+2}{2(n-2)}}\delta_{P_{2},t_{2}}^{\frac{n+2}{2(n-2)}}&\text{ if }n\geq 6.
    \end{cases}
    \label{Ayd-P1}
\end{equation}
Recalling \eqref{A16-P1}, we may use Lemma \ref{L2} to obtain
\[
|I_2| \leq \begin{cases} 
 O(1) t_1^{-1} (t_1t_2)^{-\frac{n-2}{2}} & \text{ if } n = 4,5,\\
 O(1) t_1^{-1} (t_1t_2)^{-\frac{n+2}{4}} & \text{ if } n \geq 6.
\end{cases}
\]
The conclusion follows.
\end{proof}

\begin{lem}\label{V1}
If $t_{1}t_{2}\geq e$, we have that for any $v\in H^1_r(\mathbb{S}^n)$,
\begin{equation}\label{V1-1}
 \int_{\SSphere^{n}}K\Big|(\sum\limits_{l=1}^{2}\bar{\alpha}_{l}\delta_{P_{l},t_{l}})^{\frac{4}{n-2}}-\sum\limits_{l=1}^{2}(\bar{\alpha}_{l}\delta_{P_{l},t_{l}})^{\frac{4}{n-2}}\Big|v^{2}\leq O(1)\|v\|^{2}(t_{1}t_{2})^{-1}(\ln(t_{1}t_{2}))^{\frac{2}{n}},
\end{equation}
and
\begin{align}
& \int_{\SSphere^{n}}K\Big|\Big(\sum\limits_{l=1}^{2}\bar{\alpha}_{l}\delta_{P_{l},t_{l}}\Big)^{\frac{4}{n-2}}-\bar{\alpha}_{i}^{\frac{4}{n-2}}\delta_{P_{i},t_{i}}^{\frac{4}{n-2}}\Big|\delta_{P_{i},t_{i}}|v|\nonumber\\ 
&\quad\leq O(1)\|v\|\begin{cases}
(t_{1}t_{2})^{-\frac{n-2}{2}}&\text{ if }n=4,5,\\
(t_{1}t_{2})^{-\frac{n+2}{4}}(\ln(t_{1}t_{2}))^{\frac{n+2}{2n}}&\text{ if }n\geq 6.
\end{cases}\label{V1-2}
\end{align}
\end{lem}

\begin{proof} Inequality (\ref{V1-1}) follows from \eqref{A17-Aux1}, H\"older's inequality, Sobolev's inequality and Lemma \ref{L2}. Inequality (\ref{V1-2}) follows from \eqref{Ayd-P1}, H\"older's inequality, Sobolev's inequality and Lemma \ref{L2}.

\end{proof}

\begin{lem}\label{KE1}
There exists a positive constant $C=C(n,K)$ depending on $n$ and $K$ such that if $t_{i}>C$, we have that for any $v\in H^1_r(\mathbb{S}^n)$,
\begin{equation}\label{KE1-1}
\int_{\SSphere^{n}}|K-K_{i}|\delta_{P_{i},t_{i}}^{\frac{4}{n-2}}v^{2} \leq O(1)\|v\|^{2}\begin{cases}
t_{i}^{-2}(\ln t_{i})^{\frac{1}{2}}&\text{ if }n=4,\\
t_{i}^{-2}&\text{ if }n\geq5,
\end{cases}
\end{equation}
and
\begin{equation}\label{KE1-2}
\int_{\SSphere^{n}}|K-K_{i}|\delta_{P_{i},t_{i}}^{\frac{n+2}{n-2}}|v|\leq O(1)\|v\|\begin{cases}
t_{i}^{-(n-2)}&\text{ if }n=4,5\\
t_{i}^{-4}(\ln t_{i})^{\frac{2}{3}}&\text{ if }n=6,\\
t_{i}^{-\frac{n+2}{2}}&\text{ if }n\geq7.
\end{cases}
\end{equation}
\end{lem}

\begin{proof}
By H\"older's inequality, Sobolev's inequality and a computation as in the proof of Lemma \ref{A2}, we have that 
\begin{equation*}
    \int_{\SSphere^{n}}|K-K_{i}|\delta_{P_{i},t_{i}}^{\frac{4}{n-2}}v^{2}\leq C(n)\|v\|^{2}\left(\int_{0}^{\infty}\frac{|K(t_{i}^{-1}r)-K_{i}|^{\frac{n}{2}}r^{n-1}dr}{(1+r^{2})^{n}}\right)^{\frac{2}{n}},
\end{equation*}
and 
\begin{equation*}
   \int_{\SSphere^{n}}|K-K_{i}|\delta_{P_{i},t_{i}}^{\frac{n+2}{n-2}}|v|\leq C(n)\|v\|\left(\int_{0}^{\infty}\frac{|K(t_{i}^{-1}r)-K_{i}|^{\frac{2n}{n+2}}r^{n-1}dr}{(1+r^{2})^{n}}\right)^{\frac{n+2}{2n}}.
\end{equation*}
Using condition \eqref{higher0X} for $K$, we can argue as in the proof of \eqref{R2-P} to get
\begin{equation*}
\int_{0}^{\infty}\frac{|K(t_{i}^{-1}r)-K_{i}|^{\frac{n}{2}}r^{n-1}dr}{(1+r^{2})^{n}}\leq C(n,K)\begin{cases}
t_{i}^{-4}\ln t_{i}&\text{ if }n=4,\\
t_{i}^{-n}&\text{ if }n\geq 5,
    \end{cases}
\end{equation*}
and
\begin{equation*}
    \int_{0}^{\infty}\frac{|K(t_{i}^{-1}r)-K_{i}|^{\frac{2n}{n+2}}r^{n-1}dr}{(1+r^{2})^{n}}\leq C(n,K)\begin{cases}
        t_{i}^{-\frac{2n(n-2)}{n+2}}&\text{ if }n=4,5,\\
        t_{i}^{-6}\ln t_{i}&\text{ if }n=6,\\
        t_{i}^{-n}&\text{ if }n\geq7.
    \end{cases}
\end{equation*}
Combining the above four inequalities together, we can obtain (\ref{KE1-1}) and (\ref{KE1-2}).
\end{proof}

\begin{lem}\label{KE2}
There exists a positive constant $C=C(n,K)$ depending on $n$ and $K$ such that if $t_{1}$, $t_{2}>C$, we have for any $\varphi\in E_{t}$ that 
\begin{align}\label{niuzhen}
& \int_{\SSphere^{n}}K\Big(\sum\limits_{l=1}^{2}\bar \alpha_l \delta_{P_{l},t_{l}}\Big)^{\frac{n+2}{n-2}}|\varphi|\nonumber\\
    &\qquad \leq O(1)\|\varphi\|\begin{cases}
(t_{1}t_{2})^{-\frac{n-2}{2}}+\sum\limits_{l=1}^{2}t_{l}^{-(n-2)}&\text{ if }n=4,5,\\
(t_{1}t_{2})^{-2}(\ln(t_{1}t_{2}))^{\frac{2}{3}}+\sum\limits_{l=1}^{2}t_{l}^{-4}(\ln t_{l})^{\frac{2}{3}}&\text{ if }n=6,\\
(t_{1}t_{2})^{-\frac{n+2}{4}}(\ln(t_{1}t_{2}))^{\frac{n+2}{2n}}+\sum\limits_{l=1}^{2}t_{l}^{-\frac{n+2}{2}}&\text{ if }n\geq 7.\end{cases}
\end{align}
\end{lem}

\begin{proof}
By the fact $\varphi\in E_{t}$ and a direct computation, we have that 
\begin{align*}
\int_{\SSphere^{n}}K\Big(\sum\limits_{l=1}^{2}\bar \alpha_l \delta_{P_{l},t_{l}}\Big)^{\frac{n+2}{n-2}}\varphi&=\sum\limits_{l=1}^{2}
\bar \alpha_l^{\frac{n+2}{n-2}}\int_{\SSphere^{n}}(K-K_{l})\delta_{P_{l},t_{l}}^{\frac{n+2}{n-2}}\varphi\\
&\quad\quad+\int_{\SSphere^{n}}K\Big[\Big(\sum\limits_{l=1}^{2}\bar \alpha_l \delta_{P_{l},t_{l}}\Big)^{\frac{n+2}{n-2}}-\sum\limits_{l=1}^{2}\bar \alpha_l^{\frac{n+2}{n-2}} \delta_{P_{l},t_{l}}^{\frac{n+2}{n-2}}\Big]\varphi\\
&=:I_{1}+I_{2}.
\end{align*}
By (\ref{KE1-2}), we have that 
\begin{equation*}
|I_{1}|\leq C(n,K)\|\varphi\|\sum\limits_{l=1}^{2}\begin{cases}
t_{l}^{-(n-2)}&\text{ if }n=4,5\\
t_{l}^{-4}(\ln t_{i})^{\frac{2}{3}}&\text{ if }n=6,\\
t_{l}^{-\frac{n+2}{2}}&\text{ if }n\geq7.
\end{cases}
\end{equation*}

Now we estimate $I_{2}$. By the elementary inequality
\begin{equation*}
|(a+b)^{\frac{n+2}{n-2}}-a^{\frac{n+2}{n-2}}-b^{\frac{n+2}{n-2}}|\leq (a^{\frac{4}{n-2}}+b^{\frac{4}{n-2}})\min\{a,b\}
\end{equation*}
for $a, b \geq 0$, we have that 
\begin{equation*}
    |(a+b)^{\frac{n+2}{n-2}}-a^{\frac{n+2}{n-2}}-b^{\frac{n+2}{n-2}}|^{\frac{2n}{n-2}}\leq C(n)\begin{cases}
        a^{\frac{8n}{(n-2)(n+2)}}b^{\frac{2n}{n+2}}+a^{\frac{2n}{n+2}}b^{\frac{8n}{(n-2)(n+2)}}&\text{ if }n=4,5,\\
        (ab)^{\frac{n}{n-2}}&\text{ if }n\geq 6     .
    \end{cases}
\end{equation*}
By H\"older's inequality, the above inequality and Lemma \ref{L2}, we have that 
\begin{align*}
|I_{2}|&\leq C(n,K)\|\varphi\|\begin{cases}
   \sum\limits_{i\neq j}\left(\int_{\SSphere^{n}}(\delta_{P_{i},t_{i}}^{\frac{8n}{(n-2)(n+2)}}\delta_{P_{j},t_{j}}^{\frac{2n}{n+2}}\right)^{\frac{n+2}{2n}}&\text{ if }n=4,5,\\
    \left(\int_{\SSphere^{n}}(\delta_{P_{1},t_{1}}\delta_{P_{2},t_{2}})^{\frac{n}{n-2}}\right)^{\frac{n+2}{2n}}&\text{ if }n\geq6
\end{cases}\\
&\leq C(n,K)\|\varphi\|\begin{cases}
(t_{1}t_{2})^{-\frac{n-2}{2}}&\text{ if }n=4,5,\\
(t_{1}t_{2})^{-\frac{n+2}{4}}(\ln(t_{1}t_{2}))^{\frac{n+2}{2n}}&\text{ if }n \geq6.
\end{cases}
\end{align*}
Combining the above two estimates, we can obtain (\ref{niuzhen}).
    
\end{proof}

\begin{lem}\label{lidehuang}
There exists a constant $C=C(n)>0$ depending only on $n$ such that for $t_{1}$, $t_{2}>C$ and $|\alpha_{1}-\bar{\alpha}_{1}|\leq\frac{1}{2}\bar{\alpha}_{1}$, $|\alpha_{2}-\bar{\alpha}_{2}|\leq\frac{1}{2}\bar{\alpha}_{2}$, $v \in H^1_r(\mathbb{S}^n)$ and $u = \sum\limits_{l=1}^2 \alpha_l \delta_{P_t,t_l} + v$, we have that
\begin{equation}\label{yuanmeng}
\int_{\SSphere^{n}}K\Big||u|^{\frac{4}{n-2}}-(\sum\limits_{l=1}^{2}\bar{\alpha}_{l}\delta_{P_{l},t_{l}})^{\frac{4}{n-2}}\Big|\delta_{P_{i},t_{i}}^{2}\leq O(1)(\|v\|^{\frac{4}{n-2}}+\|v\|+|\alpha-\bar{\alpha}|^{\frac{4}{n-2}}+|\alpha-\bar{\alpha}|),
\end{equation}
\begin{equation}\label{yangxiu}
\int_{\SSphere^{n}}K\Big||u|^{\frac{4}{n-2}}-(\sum\limits_{l=1}^{2}\bar{\alpha}_{l}\delta_{P_{l},t_{l}})^{\frac{4}{n-2}}\Big|\delta_{P_{i},t_{i}}\delta_{P_{j},t_{j}}\leq O(1)(\|v\|^{\frac{4}{n-2}}+\|v\|+|\alpha-\bar{\alpha}|^{\frac{4}{n-2}}+|\alpha-\bar{\alpha}|),
\end{equation}
and for $\tilde{v}$, $\varphi\in H^1_r(\mathbb{S}^n)$,
\begin{equation}\label{ranshiyou}
\int_{\SSphere^{n}}K\Big||u|^{\frac{4}{n-2}}-(\sum\limits_{l=1}^{2}\bar{\alpha}_{l}\delta_{P_{l},t_{l}})^{\frac{4}{n-2}}\Big|\delta_{P_{i},t_{i}}|\tilde{v}|\leq O(1)\|\tilde{v}\|(\|v\|^{\frac{4}{n-2}}+\|v\|+|\alpha-\bar{\alpha}|^{\frac{4}{n-2}}+|\alpha-\bar{\alpha}|),
\end{equation}
\begin{equation}\label{zhupinxiu}
\int_{\SSphere^{n}}K\Big||u|^{\frac{4}{n-2}}-(\sum\limits_{l=1}^{2}\bar{\alpha}_{l}\delta_{P_{l},t_{l}})^{\frac{4}{n-2}}\Big||\tilde{v}||\varphi|\leq O(1)\|\tilde{v}\|\|\varphi\|(\|v\|^{\frac{4}{n-2}}+\|v\|+|\alpha-\bar{\alpha}|^{\frac{4}{n-2}}+|\alpha-\bar{\alpha}|).
\end{equation}
\end{lem}

\begin{proof}
Let $\xi:=\sum\limits_{l=1}^{2}\alpha_{l}\delta_{P_{l},t_{l}}$. We will divide $|u|^{\frac{4}{n-2}}-(\sum\limits_{l=1}^{2}\bar{\alpha}_{l}\delta_{P_{l},t_{l}})^{\frac{4}{n-2}}$ into two parts, i.e., 
\begin{equation*}
|u|^{\frac{4}{n-2}}-(\sum\limits_{l=1}^{2}\bar{\alpha}_{l}\delta_{P_{l},t_{l}})^{\frac{4}{n-2}}=\left(|u|^{\frac{4}{n-2}}-\xi^{\frac{4}{n-2}}\right)+\Big[\xi^{\frac{4}{n-2}}-\Big(\sum\limits_{l=1}^{2}\bar{\alpha}_{l}\delta_{P_{l},t_{l}}\Big)^{\frac{4}{n-2}}\Big]
\end{equation*}

By the mean value theorem, we have that 
\begin{equation*}
 \Big|\xi^{\frac{4}{n-2}}-\Big(\sum\limits_{l=1}^{2}\bar{\alpha}_{l}\delta_{P_{l},t_{l}}\Big)^{\frac{4}{n-2}}\Big| \leq C(n,K)\begin{cases}
      (|\alpha-\bar{\alpha}|+ |\alpha-\bar{\alpha}|^{\frac{4}{n-2}})\sum\limits_{l=1}^{2}\delta_{P_{l},t_{l}}^{\frac{4}{n-2}}&\text{ if }n=4,5,6,\\
      |\alpha-\bar{\alpha}|(\delta_{P_{i},t_{i}}^{\frac{4}{n-2}}+\delta_{P_{i},t_{i}}^{\frac{4}{n-2}-1}\delta_{P_{j},t_{j}})&\text{ if }n\geq7.
 \end{cases}  
\end{equation*}
We also have that 
\begin{align*}
   | |u|^{\frac{4}{n-2}}-\xi^{\frac{4}{n-2}}|&\leq C(n)\begin{cases}
   \xi^{\frac{4}{n-2}-1}|v|+|v|^{\frac{4}{n-2}}&\text{ if }n=4,5,\\
   |v|^{\frac{4}{n-2}}&\text{ if }n\geq6
   \end{cases}\\
   &\leq C(n,K)\begin{cases}
   (\sum\limits_{l=1}^{2}\delta_{P_{l},t_{l}}^{\frac{6-n}{n-2}})|v|+|v|^{\frac{4}{n-2}}&\text{ if }n=4,5,\\
   |v|^{\frac{4}{n-2}}&\text{ if }n\geq6.
   \end{cases}
   \end{align*}
 Combining the above two inequalities together, we have that 
\begin{align*}
  &\quad  \Big| |u|^{\frac{4}{n-2}}-\Big(\sum\limits_{l=1}^{2}\bar{\alpha}_{l}\delta_{P_{l},t_{l}}\Big)^{\frac{4}{n-2}}\Big|\\
  &\leq C(n,K)
\begin{cases}
    (|\alpha-\bar{\alpha}|+ |\alpha-\bar{\alpha}|^{\frac{4}{n-2}})\sum\limits_{l=1}^{2}\delta_{P_{l},t_{l}}^{\frac{4}{n-2}}+(\sum\limits_{l=1}^{2}\delta_{P_{l},t_{l}}^{\frac{6-n}{n-2}})|v|+|v|^{\frac{4}{n-2}}&\text{ if }n=4,5,\\
 (|\alpha-\bar{\alpha}|+ |\alpha-\bar{\alpha}|^{\frac{4}{n-2}})\sum\limits_{l=1}^{2}\delta_{P_{l},t_{l}}^{\frac{4}{n-2}}+|v|^{\frac{4}{n-2}}&\text{ if }n=6,\\
    |\alpha-\bar{\alpha}|(\delta_{P_{i},t_{i}}^{\frac{4}{n-2}}+\delta_{P_{i},t_{i}}^{\frac{4}{n-2}-1}\delta_{P_{j},t_{j}})+|v|^{\frac{4}{n-2}}&\text{ if }n\geq7.
\end{cases}
\end{align*}  
By the above inequality, H\"older's inequality and Lemma \ref{L2}, we can obtain (\ref{yuanmeng})-(\ref{zhupinxiu}).
\end{proof}

\subsection{Estimates specific to the finite-dimensional reduction}

\begin{lem}\label{zhognshizhang}
Let $t_1, t_2 \geq 2$ and $t \mapsto v(t) \in E_t$ be $C^1$. Then
\begin{align}
&(1-\Pi_{t})(\frac{\partial v}{\partial t_{i}})\nonumber\\
&\quad=\left\langle v,\frac{\partial^{2}\delta_{P_{i},t_{i}}}{\partial t_{i}^{2}}\right\rangle\Big[C_{1}t_{j}^{-1}(t_{1}t_{2})^{-(n-3)}(1+o(1))\delta_{P_{i},t_{i}}-C_{2}t_{i}^{2}(1+o(1))\frac{\partial\delta_{P_{i},t_{i}}}{\partial t_{i}}\nonumber\\
&\quad\quad-C_{3}t_{j}^{-1}(t_{1}t_{2})^{-\frac{n-4}{2}}(1+o(1))\delta_{P_{j},t_{j}}+C_{4}(t_{1}t_{2})^{-\frac{n-4}{2}}(1+o(1))\frac{\partial\delta_{P_{j},t_{j}}}{\partial t_{j}}\Big],\label{zhongchangjiang}
\end{align}
where $C_{1}$, $C_{2}$, $C_{3}$, $C_{4}$ are positive constants depending on $n$ and $K$, and $o(1)$ denotes some scalar tending to $0$ as $t_1, t_2 \rightarrow \infty$. Moreover,
\begin{equation}\label{high}
\left|\left\langle v,\frac{\partial^{2}\delta_{P_{i},t_{i}}}{\partial t_{i}^{2}}\right\rangle\right|\leq C(n)t_{i}^{-2}\|v\|.
\end{equation}
\end{lem}

\begin{proof}

Let $V_{t}$ denote the orthogonal complement space of $E_{t}$ with respect to the $H^{1}$ inner product $\langle\cdot,\cdot\rangle$, i.e.,
\begin{equation*}
    V_{t}=\mbox{span}\{e_{1},e_{2},e_{3},e_{4}\},
\end{equation*}
where 
\begin{equation*}
e_{1}:=\delta_{P_{i},t_{i}},    
\end{equation*}
\begin{equation*}
e_{2}:=\frac{\partial \delta_{P_{i},t_{i}}}{\partial t_{i}},
\end{equation*}
\begin{equation*}
    e_{3}:=\delta_{P_{j},t_{j}}-\frac{\langle\delta_{P_{j},t_{j}},e_{1}\rangle}{\langle e_{1},e_{1}\rangle}e_{1}-\frac{\langle\delta_{P_{j},t_{j}},e_{2}\rangle}{\langle e_{2},e_{2}\rangle}e_{2},
\end{equation*}
and 
\begin{equation*}
    e_{4}:=\frac{\partial\delta_{P_{j},t_{j}}}{\partial t_{j}}-\frac{\langle\frac{\partial\delta_{P_{j},t_{j}}}{\partial t_{j}},e_{1}\rangle}{\langle e_{1},e_{1}\rangle}e_{1}-\frac{\langle\frac{\partial\delta_{P_{j},t_{j}}}{\partial t_{j}},e_{2}\rangle}{\langle e_{2},e_{2}\rangle}e_{2}-\frac{\langle\frac{\partial\delta_{P_{j},t_{j}}}{\partial t_{j}},e_{3}\rangle}{\langle e_{3},e_{3}\rangle}e_{3}.
\end{equation*}
Then $(1-\Pi_{t})(\frac{\partial v}{\partial t_{i}})$ can be represented as 
\begin{equation*}
    (1-\Pi_{t})(\frac{\partial v}{\partial t_{i}})=\frac{\langle\frac{\partial v}{\partial t_{i}},e_{1}\rangle}{\langle e_{1},e_{1}\rangle}e_{1}+\frac{\langle\frac{\partial v}{\partial t_{i}},e_{2}\rangle}{\langle e_{2},e_{2}\rangle}e_{2}+\frac{\langle\frac{\partial v}{\partial t_{i}},e_{3}\rangle}{\langle e_{3},e_{3}\rangle}e_{3}+\frac{\langle\frac{\partial v}{\partial t_{i}},e_{4}\rangle}{\langle e_{4},e_{4}\rangle}e_{4}.
\end{equation*}
Then (\ref{zhongchangjiang}) follows from the above equality and the estimates in Lemma \ref{A1} and \ref{A2}

A direct computation gives that 
\begin{align*}
 \left\langle v,\frac{\partial^{2}\delta_{P_{i},t_{i}}}{\partial t_{i}^{2}}\right\rangle&=-\int_{\SSphere^{n}}L_{g_{0}}(\frac{\partial^{2}\delta_{P_{i},t_{i}}}{\partial t_{i}^{2}})v
 =\int_{\SSphere^{n}} \frac{\partial^{2}}{\partial t_{i}^{2}} (-L_{g_{0}}\delta_{P_{i},t_{i}})v
    = \frac{n(n-2)}{4}\int_{\SSphere^{n}}\frac{\partial^{2}\delta_{P_{i},t_{i}}^{\frac{n+2}{n-2}}}{\partial t_{i}^{2}}v\\
 &=\frac{n(n+2)}{4}\int_{\SSphere^{n}}\left(\delta_{P_{i},t_{i}}^{\frac{4}{n-2}}\frac{\partial^{2} \delta_{P_{i},t_{i}}}{\partial t_{i}^{2}}+\frac{4}{n-2}\delta_{P_{i},t_{i}}^{\frac{4}{n-2}-1}(\frac{\partial\delta_{P_{i},t_{i}}}{\partial t_{i}})^{2}\right)v.
\end{align*}
Recall that 
\begin{equation*}
    \Big|\frac{\partial \delta_{P_{i},t_{i}}}{\partial t_{i}}\Big| \leq C(n)t_i^{-1} \delta_{P_i,t_i},~~    \Big|\frac{\partial^{2} \delta_{P_{i},t_{i}}}{\partial t_{i}^{2}}\Big| \leq C(n)t_i^{-2} \delta_{P_i,t_i}.
\end{equation*}
Then (\ref{high}) can be obtained by the above estimates and H\"older's inequality.
\end{proof}

\begin{lem}\label{liutengfei}
Let $t_1, t_2$, $\alpha(t), v(t)$ be as in Proposition \ref{infinite}, and $u = \sum\limits_{l=1}^2 \alpha_l \delta_{P_l,t_l} + v$. Then 
\begin{align}\label{linyang}
\left|\int_{\SSphere^{n}}K|u|^{\frac{4}{n-2}}u\frac{\partial \delta_{P_{i},t_{i}}}{\partial t_{i}}\right|
    &\leq O(1)t_{i}^{-1}\big(t_{i}^{-(n-2)}+(t_{1}t_{2})^{-\frac{n-2}{2}}\big),
\\
\label{helin}
\left|\int_{\SSphere^{n}}K|u|^{\frac{4}{n-2}}u\delta_{P_{i},t_{i}}\right|
    &\leq O(1),
\end{align}
and
\begin{align}
&\left|\int_{\SSphere^{n}}K\left(|u|^{\frac{4}{n-2}}u-(\sum\limits_{l=1}^{2}\bar{\alpha}_{l}\delta_{P_{l},t_{l}})^{\frac{n+2}{n-2}}\right)\frac{\partial \delta_{P_{i},t_{i}}}{\partial t_{i}}\right|\nonumber\\
&\quad\leq O(1)t_{i}^{-1}\begin{cases}
(t_{1}t_{2})^{-(n-2)}+\sum\limits_{l=1}^{2}t_{l}^{-2(n-2)},&n=4,5,\\
(t_{1}t_{2})^{-4}(\ln(t_{1}t_{2}))^{\frac{4}{3}}+\sum\limits_{l=1}^{2}t_{l}^{-8}(\ln t_{l})^{\frac{4}{3}},&n=6,\\
(t_{1}t_{2})^{-\frac{n+2}{2}}(\ln(t_{1}t_{2}))^{\frac{n+2}{n}}+\sum\limits_{l=1}^{2}t_{l}^{-(n+2)},&n\geq 7.\end{cases}\label{lunyang}
\end{align}

\end{lem}

\begin{proof}
Let $\xi:=\sum\limits_{l=1}^{2}\alpha_{l}\delta_{P_{l},t_{l}}$ and $\bar{\xi}:=\sum\limits_{l=1}^{2}\bar{\alpha}_{l}\delta_{P_{l},t_{l}}$.

By Proposition \ref{infinite}, we have $|\alpha - \bar \alpha| + \|v\| = o(1)$ as $t_1, t_2 \rightarrow \infty$. Thus, by \eqref{AB1-} and \eqref{A2-1},
\begin{equation}
\|\delta_{P_i,t_i}\| + |\alpha_i| + |\alpha_i|^{-1}+ \|v\| + \|u\| = O(1).
	\label{lft-P1}
\end{equation}
This implies \eqref{helin}.

For the remaining two estimates, we only need to prove \eqref{lunyang}, as \eqref{linyang} follows from \eqref{lunyang} and estimate \eqref{Axd-1} in Lemma \ref{Axd}. To prove \eqref{lunyang}, we split the integral on the left hand side of \eqref{lunyang}:
\begin{align*}
I := \int_{\SSphere^{n}}K\left(|u|^{\frac{4}{n-2}}u-\bar{\xi}^{\frac{n+2}{n-2}}\right)\frac{\partial \delta_{P_{i},t_{i}}}{\partial t_{i}}
 &=\frac{n+2}{n-2}\alpha_{i}^{\frac{4}{n-2}}\int_{\SSphere^{n}}K \delta_{P_{i},t_{i}}^{\frac{4}{n-2}}\frac{\partial \delta_{P_{i},t_{i}}}{\partial t_{i}}v\\
&\quad + \frac{n+2}{n-2}\int_{\SSphere^{n}}K\left(\xi^{\frac{4}{n-2}}-(\alpha_{i}\delta_{P_{i},t_{i}})^{\frac{4}{n-2}}\right)v\frac{\partial \delta_{P_{i},t_{i}}}{\partial t_{i}}\\
&\quad +\int_{\SSphere^{n}}K\left(|u|^{\frac{4}{n-2}}u-\xi^{\frac{n+2}{n-2}}-\frac{n+2}{n-2}\xi^{\frac{4}{n-2}}v\right)\frac{\partial \delta_{P_{i},t_{i}}}{\partial t_{i}}\\
&\quad+\int_{\SSphere^{n}}K\left(\xi^{\frac{n+2}{n-2}}-\bar{\xi}^{\frac{n+2}{n-2}} - \frac{n+2}{n-2} \bar\xi^{\frac{4}{n-2}}(\xi - \bar \xi)\right)\frac{\partial \delta_{P_{i},t_{i}}}{\partial t_{i}}\\
&\quad+ \frac{n+2}{n-2} \int_{\SSphere^{n}}K  \bar\xi^{\frac{4}{n-2}}(\xi - \bar \xi) \frac{\partial \delta_{P_{i},t_{i}}}{\partial t_{i}}\\
&=:I_{1}+I_{2}+I_{3} + I_4 + \frac{n+2}{n-2} I_5.
\end{align*}

To estimate $I_{1}$, we note that 
\begin{align*}
    I_{1}&=\frac{n+2}{n-2}\alpha_{i}^{\frac{4}{n-2}}\int_{\SSphere^{n}}(K-K_{i})\delta_{P_{i},t_{i}}^{\frac{4}{n-2}}\frac{\partial \delta_{P_{i},t_{i}}}{\partial t_{i}}v
        -\alpha_{i}^{\frac{4}{n-2}}K_{i}\int_{\SSphere^{n}}\delta_{P_{i},t_{i}}^{\frac{n+2}{n-2}}(1-\Pi_{t})(\frac{\partial v}{\partial t_{i}}).
\end{align*}
In view of \eqref{A16-P1}, the first term can be bounded using \eqref{KE1-2}. The second term can be bounded using (\ref{zhongchangjiang}), (\ref{high}), (\ref{AB1-}), (\ref{A2-1}) and (\ref{A2-2}). We obtain
\begin{equation*}
|I_{1}|\leq O(1)t_{i}^{-1}\|v\|\begin{cases}
(t_{1}t_{2})^{-(n-2)} + t_{i}^{-(n-2)}&\text{ if }n=4,5\\
(t_{1}t_{2})^{-4} + t_{i}^{-4}(\ln t_{i})^{\frac{2}{3}}&\text{ if }n=6,\\
(t_{1}t_{2})^{-(n-2)} + t_{i}^{-\frac{n+2}{2}}&\text{ if }n\geq7.
\end{cases}
\end{equation*}

In view of \eqref{lft-P1} and \eqref{A16-P1}, we can use the proof of \eqref{V1-2} to show that
\[
|I_2| \leq O(1)  t_i^{-1}\|v\| \begin{cases}
(t_{1}t_{2})^{-\frac{n-2}{2}}&\text{ if }n=4,5,\\
(t_{1}t_{2})^{-\frac{n+2}{4}}(\ln(t_{1}t_{2}))^{\frac{n+2}{2n}}&\text{ if }n\geq 6.
\end{cases}
\]

For $I_{3}$, observe that
\[
\Big|
|u|^{\frac{4}{n-2}}u
-
\xi^{\frac{n+2}{n-2}}
-
\frac{n+2}{n-2}
\xi^{\frac{4}{n-2}}v
\Big|  
\le
\begin{cases}
O(1) |v|^{\frac{n+2}{n-2}}
& \text{if } |v| \ge \frac{1}{2} \xi, \\ 
O(1) \xi^{\frac{6-n}{n-2}}|v|^2
& \text{if } |v| < \frac{1}{2}\xi,
\end{cases}
\]
and, by \eqref{A16-P1} and \eqref{lft-P1}, $|\frac{\partial \delta_{P_{i},t_{i}}}{\partial t_{i}}| \leq O(1) t_{i}^{-1} \delta_{P_{i},t_{i}} \leq O(1) t_i^{-1}\xi$. Therefore
\[
|I_3| \leq O(1) t_i^{-1} \int_{\mathbb{S}^n} \big(|v|^{\frac{2n}{n-2}} + \xi^{\frac{4}{n-2}} |v|^2\big)
	\leq O(1) t_i^{-1} \big(\|v\|^{\frac{2n}{n-2}} + \|v\|^2\big) \leq O(1) t_i^{-1} \|v\|^2.
\]

For $I_4$, since $|\alpha - \bar \alpha| = o(1)$, we have
\[
\Big|
\xi^{\frac{n+2}{n-2}}- \bar \xi^{\frac{n+2}{n-2}}
	- \frac{n+2}{n-2}\bar \xi^{\frac{4}{n-2}} (\xi - \bar \xi)
\Big|  
\le O(1) \bar \xi^{\frac{6-n}{n-2}}|\xi - \bar \xi|^2,
\]
which together with \eqref{A16-P1} implies
\[
|I_4| \leq O(1) t_i^{-1} |\alpha - \bar \alpha|^2.
\]

For $I_5$, we compute
\[
I_5 = (\alpha_i - \bar \alpha_i) \int_{\mathbb{S}^n} K \bar \xi^{\frac{4}{n-2}} \delta_{P_i,t_i} \frac{\partial \delta_{P_i,t_i}}{\partial t_i} 
	+ (\alpha_j - \bar \alpha_j) \int_{\mathbb{S}^n} K \bar \xi^{\frac{4}{n-2}} \delta_{P_j,t_j} \frac{\partial \delta_{P_i,t_i}}{\partial t_i} =: I_{51} + I_{52}.
\]
By Lemma \ref{Ayd}, 
\[
|I_{51}| \leq O(1) t_i^{-1}|\alpha - \bar \alpha| \begin{cases} 
t_i^{-(n-2)} + (t_1t_2)^{-\frac{n-2}{2}}& \text{ if } n = 4,5,\\
t_i^{-(n-2)} + (t_1t_2)^{-\frac{n+2}{4}} & \text{ if } n \geq 6.
\end{cases}
\]
By \eqref{A16-P1} and Lemma \ref{L2},
\[
|I_{52}| \leq O(1)t_i^{-1}|\alpha - \bar \alpha|\int_{\mathbb{S}^n} \Big(\delta_{P_i,t_i}^{\frac{n+2}{n-2}} \delta_{P_j,t_j} + \delta_{P_i,t_i} \delta_{P_j,t_j} ^{\frac{n+2}{n-2}}\Big) \leq O(1)t_i^{-1}|\alpha - \bar \alpha| (t_1t_2)^{-\frac{n-2}{2}}.
\]
Altogether, we have
\begin{align*}
|I| &\leq O(1) \big(|\alpha - \bar \alpha| + \|v\|\big)^2\\
		&\qquad  
		+ O(1) \big(|\alpha - \bar \alpha| + \|v\|\big) 
		\begin{cases}
		(t_{1}t_{2})^{-\frac{n-2}{2}} + t_{i}^{-(n-2)}&\text{ if }n=4,5\\
		(t_{1}t_{2})^{-2}(\ln(t_{1}t_{2}))^{\frac{2}{3}}  + t_{i}^{-4}(\ln t_{i})^{\frac{2}{3}}&\text{ if }n=6,\\
		(t_{1}t_{2})^{-\frac{n+2}{4}}(\ln(t_{1}t_{2}))^{\frac{n+2}{2n}}  + t_{i}^{-\frac{n+2}{2}}&\text{ if }n\geq7.
		\end{cases}
\end{align*}
Inserting the estimate for $|\alpha - \bar \alpha| + \|v\|$ in Proposition \ref{infinite}, we obtain \eqref{lunyang}.
 \end{proof}

\subsubsection*{Rights retention statement} For the purpose of Open Access, the authors have applied a CC BY public copyright licence to any Author Accepted Manuscript (AAM) version arising from this submission.

\end{document}